\documentclass{amsart}
\usepackage[margin=1in]{geometry}

\usepackage{color,amssymb, comment, mathrsfs,tikz,upgreek,contour, soul, colonequals, mathdots,fancyvrb}
\usepackage{tikz-cd}
\usepackage{cleveref}
\usetikzlibrary{cd}
\usetikzlibrary{fit,shapes.geometric}
\usetikzlibrary{matrix}
\usetikzlibrary{arrows}
\usepackage{enumitem}
\usepackage{mathtools}
\usepackage[all]{xy}
\usepackage[mathlines,pagewise]{lineno}
\allowdisplaybreaks

\newtheorem{theorem}{Theorem}[section]
\newtheorem{lemma}[theorem]{Lemma}
\newtheorem{proposition}[theorem]{Proposition}
\newtheorem{corollary}[theorem]{Corollary}

\usepackage[only,llbracket,rrbracket,llparenthesis,rrparenthesis]{stmaryrd} 
\usepackage{accsupp}

\theoremstyle{definition}

\newtheorem{example}[theorem]{Example}

\newtheorem{notation}[theorem]{Notation}
\newtheorem{convention}[theorem]{Convention}

\theoremstyle{remark}
\newtheorem{remark}[theorem]{Remark}

\newtheorem{chunk}[theorem]{}

\numberwithin{equation}{section}

\newcommand{\kk}{\Bbbk}

\newcommand{\del}{\partial}

\newcommand{\xra}{\xrightarrow}

\newcommand{\Hom}{\operatorname{Hom}}

\newcommand{\Tor}{\operatorname{Tor}}

\newcommand{\q}{\mathbf{q}}
\newcommand{\m}{\mathfrak{m}}

\newcommand{\pf}[2]{\mathrm{pf}_{\overline{#1}}(#2)}

\newcommand{\sgn}{\textnormal{sgn}}
\newcommand{\ee}{\mathsf{e}}
\newcommand{\ff}{\mathsf{f}}
\renewcommand{\gg}{\mathsf{g}}

\newcommand{\ov}{\overline}
\newcommand{\rank}{\mathrm{rank}}

\makeatletter
\DeclareFontEncoding{LS2}{}{\@noaccents}
\makeatother
\DeclareFontSubstitution{LS2}{stix}{m}{n}

\DeclareSymbolFont{largesymbolsstix}{LS2}{stixex}{m}{n}

\DeclareMathDelimiter{\lbrbrak}{\mathopen}{largesymbolsstix}{"EE}{largesymbolsstix}{"14}
\DeclareMathDelimiter{\rbrbrak}{\mathclose}{largesymbolsstix}{"EF}{largesymbolsstix}{"15}

\crefname{diagram}{diagram}{diagrams}
\crefname{diagram}{Diagram}{Diagrams}
\creflabelformat{diagram}{(#1) #2 #3}

\begin{document}

\title[The Tor algebra of trimmings of Gorenstein ideals]{The Tor algebra of trimmings of Gorenstein ideals}

\author[L.~Ferraro]{Luigi Ferraro}
\address{School of Mathematical and Statistical Sciences,
University of Texas Rio Grande Valley, Edinburg, TX 78539, U.S.A.}
\email{luigi.ferraro@utrgv.edu}

\author[A.~Hardesty]{Alexis Hardesty}
\address{Division of Mathematics,
Texas Woman's University, Denton, TX 76204, U.S.A.}
\email{ahardesty1@twu.edu}


\begin{abstract}
Let $(R,\m,\kk)$ be a regular local ring of dimension 3. Let $I$ be a Gorenstein ideal of $R$ of grade 3. Buchsbaum and Eisenbud proved that there is a skew-symmetric matrix of odd size such that $I$ is generated by the sub-maximal pfaffians of this matrix. Let $J$ be the ideal obtained by multiplying some of the pfaffian generators of $I$ by $\m$; we say that $J$ is a trimming of $I$. Building on a recent paper of Vandebogert, we construct an explicit free resolution of $R/J$ and compute a partial DG algebra structure on this resolution. We provide the full DG algebra structure in the appendix. We use the products on this resolution to study the Tor algebra of such trimmed ideals and we use the information obtained to prove that recent conjectures of Christensen, Veliche and Weyman on ideals of class $\mathbf{G}$ hold true in our context. Furthermore, we address the realizability question for ideals of class $\mathbf{G}$.
\end{abstract}

\maketitle

\section{Introduction}

Let $(R,\m,\kk)$ be a regular local ring and $I$ a perfect ideal of grade $3$. Buchsbaum and Eisenbud showed in \cite{BuchEisen} that a minimal free resolution $F_\bullet$ of $R/I$ has a differential graded (DG) algebra structure. This DG algebra induces a graded algebra structure on $\Tor^R_\bullet(R/I,\kk)\colonequals\mathrm{H}_\bullet(F_\bullet\otimes_R\kk)$. By results of Weyman \cite{Weyman} and of Avramov, Kustin, and Miller \cite{AKM}, this graded algebra structure does not depend on the DG algebra structure on $F_\bullet$ and may be classified into one of five distinct classes. 

Of particular interest to us are ideals whose Tor algebra is of class $\mathbf{G}(r)$, where $r$ is a parameter defined as the rank of the multiplication map
\[
\Tor^R_2(R/I,\kk)\rightarrow\Hom_\kk(\Tor^R_1(R/I,\kk),\Tor^R_3(R/I,\kk));
\]
see \cite[1.1]{Avramov2012}. Every Gorenstein ideal of grade 3 that is not a complete intersection ideal is of class $\mathbf{G}(r)$. Moreover, by a result of Buchsbaum and Eisenbud \cite{BuchEisen}, every Gorenstein ideal of grade 3 is generated by the sub-maximal pfaffians of a skew-symmetric matrix of odd size. It was conjectured by Avramov in \cite{Avramov2012} that every ideal of class $\mathbf{G}$ is Gorenstein. A counterexample to this conjecture was found by Christensen and Veliche in \cite{CVExamples} and later by Christensen, Veliche and Weyman in \cite{trimming}. The counterexample in \cite{trimming} was constructed by a process referred to as ``trimming", which consists of replacing a generator $g$ of the ideal by $\m g$.

In \cite{Linkage}, Christensen, Veliche and Weyman make the following conjectures: let $J$ be a grade 3 perfect ideal of class $\mathbf{G}(r)$ that is not Gorenstein, minimally generated by $m$ elements, then
\begin{enumerate}
\item if the type of $J$ is 2, then $2\leq r\leq m-5$ or $r=m-3$.
\item If the type of $J$ is at least 3, then $2\leq r\leq m-4$.
\end{enumerate}
Using results from \cite{trimming}, these conjectures can be refined for ideals $J$ of class $\mathbf{G}(r)$ arising by trimming a Gorenstein ideal in a regular local ring of dimension 3 (see \Cref{rem:conj}). In \Cref{cor:conj}, we prove these conjectures hold for ideals obtained by trimming the pfaffian generators of a Gorenstein ideal in a regular local ring of dimension 3.

The main tool used in this paper to prove the conjectures mentioned above is a DG algebra resolution of the trimmed ideal $J$. This DG algebra resolution was originally constructed by Vandebogert in \cite{DGTrimmed} in a more general setting. The product formulae given in \cite[Theorem 3.2]{DGTrimmed} are defined implicitly by lifting certain cycles through the differential of the resolution. In this paper we make use of Vandebogert's construction and of pfaffian identities proved by Knuth \cite{Knuth} to give more explicit products on the resolution constructed by Vandebogert.

The paper is organized as follows. In Section 2 we recall background information on perfect ideals of grade 3 and their classification, on DG algebra resolutions of Gorenstein ideals, and on pfaffians and their identities. In Section 3 we apply Vandebogert's construction to give a more explicit free resolution of an ideal obtained by trimming pfaffian generators of a Gorenstein ideal of grade 3. In Section 4 we use the product table provided by Vandebogert and the pfaffian identities provided by Knuth to give a more explicit (partial) DG algebra structure on the free resolution constructed in Section 3. In Section 5 we use the resolution and the products constructed in Section 4 to study the Tor algebra of ideals arising by trimming the pfaffian generators of a Gorenstein ideal of grade 3. This section contains the main result of the paper, i.e., \Cref{thm:TrimClass}. We use these results to prove that the conjectures of Christensen, Veliche and Weyman hold true in our context. Finally, in Section 6 we address the realizability question for ideals of class $\mathbf{G}(r)$, namely, if one fixes the Betti numbers of $R/I$, what values of $r$ are realized by an ideal with those Betti numbers and of class $\mathbf{G}(r)$? Novel examples are provided in \Cref{prop:OddEx} and \Cref{prop:EvenEx}. The reader may be interested in knowing the full DG algebra structure on the resolution of a trimming of a Gorenstein ideal, which is computed in the Appendix. 

\section{Background and Notation}

Throughout the paper $(R,\m,\kk)$ denotes a regular local ring of dimension 3. The results of this paper also hold true if $R$ is a three dimensional standard graded polynomial ring with coefficients in $\kk$.

\begin{chunk}\label{chk:ProdTables}
Let $I$ be a perfect ideal of $R$ of grade 3. We say that $I$ has \emph{format} $(1,m,m+n-1,n)$ if the minimal free resolution $F_\bullet$ of $R/I$ is of the form
\[
F_\bullet:0\longrightarrow R^n\longrightarrow R^{m+n-1}\longrightarrow R^m\longrightarrow R\longrightarrow0.
\]
We fix bases
\begin{equation}\label{basis}
\{e_i\}_{i=1,\ldots,m},\quad\{f_i\}_{i=1,\ldots,m+n-1},\quad\{g_i\}_{i=1,\ldots,n}
\end{equation}
of $F_1,F_2$ and $F_3$ respectively.

It was proved in \cite{BuchEisen} that a free resolution of length 3 always admits a DG algebra structure. The DG algebra structure on $F_\bullet$ induces a graded $\kk$-algebra structure on $\mathrm{H}(F_\bullet\otimes_R\kk)$, which is denoted by $\Tor^R(R/I,\kk)$. 

It was shown in \cite{AKM} that, even though the DG algebra structure of $F_\bullet$ is not unique, the algebra structure on $\Tor^R(R/I,\kk)$ is unique. Moreover, the bases in \eqref{basis} can be chosen such that the $\kk$-bases induced in $\Tor^R(R/I,\kk)$, which we denote by
\[
\{\ee_i\}_{i=1,\ldots,m},\quad\{\ff_i\}_{i=1,\ldots,m+n-1},\quad\{\gg_i\}_{i=1,\ldots,n},
\]
have products described by one of the multiplicative structures in the product table below:
\begin{center}
\begin{tabular}{rlrl}
    $\mathbf{C}(3)$ & $\ee_1 \ee_2 = \ff_3$, $\ee_2 \ee_3 = \ff_1$, $\ee_3 \ee_1 = \ff_2$ & $\ee_i \ff_i = \gg_1$ & for $1\leq i\leq 3$\\
    $\mathbf{T}$ & $\ee_1 \ee_2 = \ff_3$, $\ee_2 \ee_3 = \ff_1$, $\ee_3 \ee_1 = \ff_2$\\
    $\mathbf{B}$ & $\ee_1 \ee_2 = \ff_3$ & $\ee_i \ff_i = \gg_1$ & for $1\leq i \leq 2$\\
    $\mathbf{G}(r)$ &  & $\ee_i \ff_i = \gg_1$ & for $1\leq i\leq r$\\
    $\mathbf{H}(p,q)$ & $\ee_i \ee_{p+1} = \ff_i$ for $1\leq i\leq p$ & $\ee_{p+1}\ff_{p+j} = \gg_j$ & for $1\leq j\leq q$,
\end{tabular}
\end{center}
with $r,p,q$ nonnegative integers and $r\geq2$. The products not listed are either zero or can be deduced from the ones listed by graded-commutativity. Depending on the product structure of $\Tor^R(R/I,\kk)$, we say that the ideal $I$ belongs to one of the following classes: $\mathbf{C}(3)$, $\mathbf{T}$, $\mathbf{B}$, $\mathbf{G}(r)$, and $\mathbf{H}(p,q)$.
\end{chunk}

\begin{chunk}
Let $T=(T_{i,j})$ be a $m\times m$ skew-symmetric matrix with entries in $R$ and with zeros on the diagonal. We define a function $\mathcal{P}$ from the set of words in the letters $\{1,\ldots,m\}$ to $R$ as follows
\[
\mathcal{P}[i,j] \colonequals T_{i,j},\quad \mathrm{for}\;i,j\in\{1,\ldots,m\}
\]
\[
\mathcal{P}[i_1\ldots i_n] \colonequals
\begin{cases}
\sum\mathrm{sgn}\binom{i_1\ldots i_{2k}}{j_1\cdots j_{2k}}\mathcal{P}[j_1j_2]\cdots\mathcal{P}[j_{2k-1}j_{2k}]&\mathrm{if}\;n=2k\\
0&\mathrm{otherwise},\end{cases}
\]
where $\mathrm{sgn}$ denotes the sign of a permutation and the sum is taken over all the partitions of $\{i_1,\ldots,i_{2k}\}$ in $k$ subsets of size 2.
By convention we set $\mathcal{P}[\emptyset]=1$ and we set $\mathcal{P}[w]=0$ if $w$ has a repeated letter. 

If $\{i_1,\ldots,i_n\}\subseteq\{1,\ldots,m\}$ with $i_1<\cdots<i_n$, we define the \emph{pfaffian} of the submatrix of $T$ obtained by only considering the rows and columns in positions $i_1,\ldots,i_n$ to be
\[
\mathrm{pf}_{i_1,\ldots,i_n}(T)\colonequals\mathcal{P}[i_1,\ldots,i_n].
\]
We also set a notation for the pfaffian of the submatrix of $T$ obtained by removing the rows and columns in positions $j_1,\ldots,j_n$
\[
\pf{j_1,\ldots,j_n}{T}\colonequals\mathrm{pf}_{\{1,\ldots,m\}\backslash\{j_1,\ldots,j_n\}}(T),
\]
and we set this quantity to be zero if the indices $j_1,\ldots,j_n$ are not distinct.

\end{chunk}

\begin{chunk}\label{ch:GorRes}
Let $I$ be a Gorenstein ideal of $R$ of grade 3 that is not a complete intersection. It was shown in \cite{BuchEisen} that there is a skew-symmetric matrix $T$ of odd size $m$, with zeros on the diagonal and entries in $\m$, such that $I=((-1)^{i+1}\pf{i}{T})_{i=1,\ldots,m}$. Moreover, a minimal free resolution of $R/I$ is given by 
\[
F_\bullet:0\rightarrow R\xra{D_3}R^m\xra{D_2}R^m\xra{D_1}R\rightarrow0,
\]
where $D_1=\begin{pmatrix}\pf{1}{T}&\cdots&(-1)^{i+1}\pf{i}{T}&\cdots&\pf{m}{T}\end{pmatrix}$, $D_2=T$ and $D_3=D_1^*$, the transpose of $D_1$.

A DG algebra structure on this resolution has been found in \cite{small}. Let $\{e_i\}_{i=1,\ldots,m},\{f_i\}_{i=1,\ldots,m}$ and $\{g\}$ be bases in degrees $1,2$ and $3$, respectively, of the resolution $F_\bullet$ with the differentials described by the matrices $D_1,D_2,D_3$ above. We denote by $(m)$ the sequence $(1\;2\;\cdots\;m)$, and by $(m)\backslash \{i_1,\ldots,i_n\}$, the sequence $(m)$ with the elements $i_1,\ldots,i_n$ removed. 
A product on $F_\bullet$ is given by the following formulae
\begin{equation}\label{eq:ProdEE}
e_ie_j=\sum_{r=1}^m(-1)^{i+1}\mathrm{sgn}\binom{(m)\backslash \{i\}}{jr(m)\backslash\{i,j,r\}}\pf{i,j,r}{T}f_r,\quad\mathrm{for}\;i<j,
\end{equation}
\[
e_if_j=\delta_{i,j}g,
\]
where $\delta_{i,j}$ is the Kronecker delta.
\end{chunk}

\begin{chunk}
For the convenience of the reader we record properties of pfaffians. We first recall the formula in \cite[(A.1.5)]{aci}, if $\beta$ is a word in $1,\ldots,m$ and $b$ an element of $\{1,\ldots,m\}$, then
\begin{equation}\label{eq:A15}
\mathcal{P}[\beta]=\sum_{r\in\beta}\mathrm{sgn}\binom{\beta}{br\beta\backslash\{b,r\}}\mathcal{P}[\beta\backslash\{b,r\}]\mathcal{P}[br].
\end{equation}
Let $1\leq i,j\leq m$ with $i\neq j$, by taking $\beta=(m)\backslash\{i\}$ and $b=j$ in \eqref{eq:A15} one gets
\begin{equation}\label{eq:singlepf}
\pf{i}{T}=\sum_{r=1}^m\mathrm{sgn}\binom{(m)\backslash\{i\}}{jr(m)\backslash\{i,j,r\}}T_{j,r}\pf{i,j,r}{T}.
\end{equation}

Let $1\leq i,j,k\leq m$ be distinct indices. Set $\beta = k(m)\backslash\{i,j\}$. We point out that $k$ appears twice in $\beta$. In terms of permutations signs, we treat the two appearances of $k$ as being distinct and set $b$ to be the first $k$ in $\beta$. Applying \eqref{eq:A15} one gets
\begin{equation}\label{eq:old3pf0}
    \sum_{r=1}^m  \mathrm{sgn}\binom{k(m)\backslash\{i,j\}}{kr(m)\backslash\{i,j,r\}}T_{k,r} \pf{i,j,r}{T} =0.
\end{equation}

Let $1\leq i,j,r,k\leq m$ be distinct indices. By taking $\beta = (m)\backslash\{i,j,r\}$ and $b=k$ in \eqref{eq:A15} one gets
\begin{equation}\label{eq:triplepf}
    \pf{i,j,r}{T} = \sum_{h=1}^m \mathrm{sgn}\binom{(m)\backslash\{i,j,r\}}{kh(m)\backslash\{i,j,r,k,h\}} T_{k,h} \pf{i,j,r,h,k}{T}.
\end{equation}

\end{chunk}

\begin{chunk}
Let $\sigma$ be a permutation in the symmetric group $S_n$. We consider the matrix
\[
\begin{pmatrix}
1&2&\cdots&n\\
1&2&\cdots&n
\end{pmatrix},
\]
and we draw a line connecting $i$ in the first row with $\sigma(i)$ in the second row for each $i$. The number of pairs of lines intersecting is called the \emph{crossing number} of $\sigma$, and we denote it by $c(\sigma)$. It is well known, see for example \cite[Pages 27-29]{Enum}, that
\[
\mathrm{sgn}(\sigma)=(-1)^{c(\sigma)}.
\]
\end{chunk}

\begin{chunk}
We recall that the \emph{unit step function}, also known as the \emph{Heaviside step function}, is the function $\theta:\mathbb{R}\backslash\{0\}\rightarrow\{0,1\}$ defined as
\[
\theta(x)=\begin{cases}
0\quad\mathrm{if}\;x<0\\
1\quad\mathrm{if}\;x>0.
\end{cases}
\]
Let $m$ be a positive integer and let $i,j,r$ be distinct elements of $\{1,\ldots,m\}$. We introduce the following notation
\[
\sigma_{i,j,r}\colonequals (-1)^{i+1}\sgn\binom{(m)\backslash{\{i\}}}{jr(m)\backslash{\{i,j,r\}}}.
\]
If $i,j,r$ are not distinct, then we set $\sigma_{i,j,r}=0$. Using the crossing number of the permutation above, it is easy to check that $\sigma_{i,j,r}$ can be written as


\begin{equation}\label{eq:Sigma3Theta}
\sigma_{i,j,r} = (-1)^{i+j+r+1+\theta(r-i)+\theta(r-j)+\theta(j-i)}.
\end{equation}
The following relations satisfied by $\sigma_{i,j,r}$ follow from the previous expression; we record them here for later reference
\begin{equation}\label{eq:Sig3Rel}
\sigma_{i,j,r} = -\sigma_{j,i,r},\quad\sigma_{i,j,r} = \sigma_{j,r,i} = \sigma_{r,i,j}.
\end{equation}
\end{chunk}

\begin{lemma}
Let $i,j,k$ be distinct elements of $\{1,\ldots,m\}$. Then
\begin{equation}\label{eq:3pf0}
    \sum_{r=1}^m  \sigma_{i,j,r}T_{k,r} \pf{i,j,r}{T} =0.
\end{equation}
\end{lemma}

\begin{proof}
Using crossing numbers, the sign in \eqref{eq:old3pf0} is $(-1)^{r+1+\theta(r-i)+\theta(r-j)}$, which is $(-1)^{i+j+\theta(j-i)}\sigma_{i,j,r}$. Therefore, by multiplying \eqref{eq:old3pf0} by $(-1)^{i+j+\theta(j-i)}$ one gets \eqref{eq:3pf0}.
\end{proof}

We introduce the following notation for the sign in \eqref{eq:triplepf} 
\[
\sigma_{i,j,r,h,k} \colonequals \sgn \binom{(m)\backslash{\{i,j,r\}}}{kh(m)\backslash{\{i,j,r,k,h\}}}.
\]
Again, if $i,j,r,h,k$ are not distinct, then we set $\sigma_{i,j,r,h,k}=0$. Using the crossing number of the permutation above one can check that the following equality holds
\begin{align}\label{eq:Sigma5Theta}
    \sigma_{i,j,r,h,k} &= (-1)^{h+k+1+\theta(k-i)+\theta(k-j)+\theta(k-r)+\theta(k-h)+\theta(h-i)+\theta(h-j)+\theta(h-r)}.
\end{align}
We point out that $\sigma_{i,j,r,h,k}$ is independent of the order of $i,j,r$.

\begin{lemma}
Let $i,h,s,k,j\in\{1,\ldots,m\}$ with $i\neq j$, $h\neq j$, $s\neq j$, and $k\neq j$. Then
\begin{equation}\label{eq:5pf0}
\sum_{r=1}^m \sigma_{i,r,h} \sigma_{i,r,h,s,k} T_{j,r} \pf{i,r,h,s,k}{T}=0.
\end{equation}
\end{lemma}

\begin{proof}
Let $\beta=j(m)\backslash\{i,h,s,k\}$. Applying \eqref{eq:A15} with $b$ equals to the first appearance of $j$ in $\beta$ yields
\begin{align*}
    0 &= \mathcal{P}[\beta]\quad \textnormal{ since $\beta$ has a repeated $j$}\\
    &= \sum_{r=1}^m \sgn\left(\begin{smallmatrix}
    \beta\\
    jr\beta\setminus{\{j,r\}}
    \end{smallmatrix}\right) T_{j,r} \mathcal{P}[\beta\setminus{\{j,r\}}]\quad \textnormal{ by}\;\eqref{eq:A15}\\
    &= \sum_{r=1}^m \sgn\left(\begin{smallmatrix}
    j(m)\setminus{\{i,h,s,k\}}\\
    jr(m)\setminus{\{i,h,s,k,r\}}
    \end{smallmatrix}\right) T_{j,r} \pf{i,h,s,k,r}{T}\\
    &= \sum_{r=1}^m \sgn\left(\begin{smallmatrix}
    (m)\setminus{\{i,h,s,k\}}\\
    r(m)\setminus{\{i,h,s,k,r\}}
    \end{smallmatrix}\right) T_{j,r} \pf{i,h,s,k,r}{T}\\
    &= \sum_{r=1}^m (-1)^{r+1+\theta(r-i)+\theta(r-h)+\theta(r-s)+\theta(r-k)} T_{j,r} \pf{i,h,s,k,r}{T}\quad\mathrm{using\;the\;crossing\;number}\\
    &= \sum_{r=1}^m (-1)^{r+1+\theta(r-i)+\theta(r-h)+\theta(r-s)+\theta(r-k)} T_{j,r} \pf{i,h,s,k,r}{T}
\end{align*}
When multiplying the sign in the last equality by $(-1)^{i+h+s+k+\theta(h-i)+\theta(k-i)+\theta(k-h)+\theta(k-s)+\theta(s-i)+\theta(s-h)},$ we obtain the sign $\sigma_{i,r,h} \sigma_{i,r,h,s,k}$. Therefore, multiplying the last equality by this power of $-1$ gives \eqref{eq:5pf0}.
\end{proof}

\section{A free resolution of a trimmed pfaffian ideal}

We fix a generating set for the maximal ideal $\m = (z_1,z_2,z_3)$. Let $I$ be a Gorenstein ideal of $R$ of grade 3 that is not a complete intersection. It has been proved in \cite[Theorem 2.1]{BuchEisen} that there is a odd sized skew-symmetric $m\times m$ matrix $T$ with entries in $\m$ such that $I=(y_1,\ldots,y_m)$ where $y_i\colonequals(-1)^{i+1}\pf{i}{T}$ for $i=1,\ldots,m$. Let $t$ be an integer such that $1\leq t\leq m$, and let $J\colonequals y_1\m+\cdots y_t\m+(y_{t+1},\ldots,y_m)$; following the terminology set in \cite{trimming} we say that $J$ is obtained from $I$ by trimming the first $t$ generators of $I$. The aim for this section is to obtain a free resolution of $R/J$.

In \cite[Theorem 2.6]{DGTrimmed}, Vandebogert constructs a free resolution of a trimmed ideal $J$ starting from the free resolution of a (general) ideal $R/I$. We apply his construction to the resolution of a Gorenstein ideal of grade 3 given in \ref{ch:GorRes}.

Let $(G_\bullet,\delta_\bullet)$ be the Koszul resolution of $k$ over $R$, with
\[
G_1=Ru_1\oplus Ru_2\oplus Ru_3,\quad G_2=Rv_{1,2}\oplus Rv_{1,3}\oplus Rv_{2,3},\quad G_3=Rw,
\]
\[
\delta_1=\begin{pmatrix}z_1&z_2&z_3\end{pmatrix},\quad\delta_2=\begin{pmatrix}-z_2&-z_3&0\\z_1&0&-z_3\\0&z_1&z_2\end{pmatrix},\quad\delta_3=\begin{pmatrix}z_3\\-z_2\\z_1\end{pmatrix}.
\]
We will need $t$ copies of $G_\bullet$, which we will denote by $(G_\bullet^k,\delta^k_\bullet)$ for $k=1,\ldots, t$. We denote the generators of $G_1^k$ by $u_1^k,u_2^k,u_3^k$, and similarly for $G_2^k,G_3^k$. We set $v^k_{\beta,\alpha}=-v^k_{\alpha,\beta}$ for $\alpha<\beta$ and $\alpha,\beta\in\{1,2,3\}$.

Let $c_{i,j,l}$, for $i,j=1,\ldots,m$ and $l=1,2,3$, be elements of $R$ satisfying the following equality 
\begin{equation}\label{eq:Tji}
T_{j,i} = \sum_{l=1}^3 c_{i,j,l}z_l.
\end{equation}

We define constants that will be used in the statement of the main theorem of this section. Let $\alpha,\beta\in\{1,2,3\}$ with $\alpha<\beta$, and let $k=1,\ldots,t$, then we define
\begin{equation}\label{eq:d^1}
d^k_{\alpha,\beta} \colonequals \sum_{i=1}^m\sum_{r=1}^m \sigma_{i,k,r} c_{i,k,\beta} c_{r,k,\alpha} \pf{i,k,r}{T}. 
\end{equation}

We use the notation set in \ref{ch:GorRes} for the minimal free resolution $(F_\bullet,D_\bullet)$ of $R/I$.

We define $F_1'$ by the decomposition $F_1 = (\oplus_{i=1}^t Re_i) \oplus  F_1'$. Similarly, we define $D_2'$ by the decomposition 
\[
D_2\in\Hom_R(F_2,F_1) = \left(\oplus_{i=1}^t \Hom_R(F_2, Re_i)\right) \oplus \Hom_R(F_2,F_1')
\]
by setting $D_2 = D_0^1 + D_0^2 + \dots + D_0^t + D_2'$.

We denote the composition 
\[
\begin{tikzpicture}[baseline=(current  bounding  box.center)]
 \matrix (m) [matrix of math nodes,row sep=3em,column sep=4em,minimum width=2em] {
F_2& Re_i & R\\};
\path[->] (m-1-1) edge  node[above]{$D_0^i$} (m-1-2);
\path[->] (m-1-2) edge  node[above]{} (m-1-3);
\end{tikzpicture}
\]
as ${D_0^i}'$ for $1\leq i \leq t$, where the second map sends $e_i\mapsto 1$.

\begin{theorem}\label{thm:res} Using the notation set so far, a free resolution of $R/J$ is given by
\begin{center}

\begin{tikzpicture}[baseline=(current  bounding  box.center)]
 \matrix (m) [matrix of math nodes,row sep=3em,column sep=4em,minimum width=2em] {
0&F_3\oplus(\oplus_{k=1}^tG_3^k)&F_2\oplus(\oplus_{k=1}^tG_2^k)&F_1^\prime\oplus(\oplus_{k=1}^tG_1^k)&R\\};
\path[->] (m-1-1) edge (m-1-2);
\path[->] (m-1-2) edge  node[above]{$\partial_3$} (m-1-3);
\path[->] (m-1-3) edge  node[above]{$\partial_2$} (m-1-4);
\path[->] (m-1-4) edge  node[above]{$\partial_1$} (m-1-5);
\end{tikzpicture}

$\del_3 = \begin{pmatrix}
D_3 & 0\\
Q_2 & \oplus_{k=1}^t \delta_3^k
\end{pmatrix}$ ,
$\del_2 = \begin{pmatrix}
D_2' & 0\\
-Q_1 & \oplus_{k=1}^t \delta_2^k
\end{pmatrix}$,
$\del_1 = \begin{pmatrix}
D_1|_{F_1} & -\sum_{k=1}^t y_k\delta_1^k
\end{pmatrix}$
\end{center}
where 
\[
Q_1=\begin{pmatrix}q_1^1\\\vdots\\q_1^t\end{pmatrix},\quad Q_2=\begin{pmatrix}q_2^1\\\vdots\\q_2^t\end{pmatrix},
\]
and the maps $q_1^k:F_2\rightarrow G_1^k,q_2^k:F_3\rightarrow G_2^k$ for $k=1,\ldots,t$ are defined by
\[
q_1^k(f_i) = \sum_{l=1}^3 c_{i,k,l}u_l^k,
\]
\[
q_2^k(g) = d^k_{1,2}v^k_{1,2} + d^k_{1,3}v^k_{1,3} + d^k_{2,3}v^k_{2,3}.
\]
\end{theorem}
Before showing the proof of \Cref{thm:res} we provide an example.

\begin{example}\label{ex:FreeRes}
Let $R=\mathbb{F}_2[x,y,z]$. In the notation set so far $z_1$ would be $x$, $z_2$ would be $y$, and $z_3$ would be $z$. Consider the matrix 
\[
T=\begin{pmatrix} 0 & 0 & 0 & x & z\\ 0 & 0 & x & z & y\\ 0 & x & 0 & y & 0\\ x & z & y & 0 & 0\\ z & y & 0 & 0 & 0 \end{pmatrix}.
\]  
Then 
\[
\pf{1}{T}=y^2, \quad \pf{2}{T}=yz, \quad \pf{3}{T}=xy+z^2, \quad
\pf{4}{T}=xz, \quad \pf{5}{T}=x^2.
\]
Let $I$ be the ideal generated by the pfaffians above. By \cite[(3.3) Proposition]{trimming}, the ideal $I$ is a perfect Gorenstein ideal of grade 3. Let $F_\bullet$ be the minimal free resolution of $R/I$ as defined in \Cref{ch:GorRes}. Let $J$ be the ideal obtained from $I$ by trimming its first generator, i.e.
\[
J=(xy^2,y^3,zy^2,yz,xy+z^2,xz,x^2).
\]
Equation \eqref{eq:Tji} allows us to find the constants $c_{i,j,l}$, for example
\[
x=T_{4,1}=c_{1,4,1}x+c_{1,4,2}y+c_{1,4,3}z,
\]
therefore $c_{1,4,1}=1,c_{1,4,2}=0,c_{1,4,3}=0$. Similarly, one can find all the remaining constants. The only nonzero ones are
\[
c_{1,4,1}=c_{2,3,1}=c_{2,4,2}=c_{3,4,2}=c_{1,5,3}=c_{2,4,3}=1,
\]
and the ones obtained by switching the first two indices. Therefore
\[
Q_1=q_1^1=\begin{pmatrix} c_{1,1,1}&c_{2,1,1}&c_{3,1,1}&c_{4,1,1}&c_{5,1,1}\\
c_{1,1,2}&c_{2,1,2}&c_{3,1,2}&c_{4,1,2}&c_{5,1,2}\\
c_{1,1,3}&c_{2,1,3}&c_{3,1,3}&c_{4,1,3}&c_{5,1,3}\end{pmatrix}=
\begin{pmatrix} 0&0&0&1&0\\0&0&0&0&0\\0&0&0&0&1\end{pmatrix}.
\]
To write down $Q_2$ we need the following constants
\begin{align*}
d_{1,2}^1&=\sum_{i=1}^5\sum_{r=1}^5c_{i,1,2}c_{r,1,1}\pf{i,1,r}{T}=0\\
d_{1,3}^1&=\sum_{i=1}^5\sum_{r=1}^5c_{i,1,3}c_{r,1,1}\pf{i,1,r}{T}=c_{5,1,3}c_{4,1,1}\pf{5,1,4}{T}=x\\
d_{2,3}^1&=\sum_{i=1}^5\sum_{r=1}^5c_{i,1,3}c_{r,1,2}\pf{i,1,r}{T}=0.
\end{align*}
Therefore the map $Q_2$ is
\[
Q_2=\begin{pmatrix}0\\x\\0\end{pmatrix}.
\]
Therefore the resolution of $R/J$ given in the theorem is
\[
\begin{tikzpicture}[baseline=(current  bounding  box.center)]
 \matrix (m) [matrix of math nodes,row sep=3em,column sep=4em,minimum width=2em] {
0&R^2&R^8&R^7&R\\};
\path[->] (m-1-1) edge (m-1-2);
\path[->] (m-1-2) edge  node[above]{$\partial_3$} (m-1-3);
\path[->] (m-1-3) edge  node[above]{$\partial_2$} (m-1-4);
\path[->] (m-1-4) edge  node[above]{$\partial_1$} (m-1-5);
\end{tikzpicture}
\]

where
\begin{align*}
\partial_1=\begin{pmatrix}yz&xy+z^2&xz&x^2&xy^2&y^3&y^2z\end{pmatrix},
\partial_2&=\begin{pmatrix}0&0&x&z&y&0&0&0\\
0&x&0&y&0&0&0&0\\
x&z&y&0&0&0&0&0\\
z&y&0&0&0&0&0&0\\
0&0&0&1&0&y&z&0\\
0&0&0&0&0&x&0&z\\
0&0&0&0&1&0&x&y\end{pmatrix},
\partial_3&=\begin{pmatrix}y^2&0\\
yz&0\\
xy+z^2&0\\
xz&0\\
x^2&0\\
0&z\\
x&y\\
0&x\end{pmatrix}.
\end{align*}
\end{example}

\begin{proof}[Proof of \Cref{thm:res}]
After checking that the choices of $q_1^k$ and $q_2^k$ above give commutative diagrams
\[
\begin{tikzpicture}[baseline=(current  bounding  box.center)]
 \matrix (m) [matrix of math nodes,row sep=3em,column sep=4em,minimum width=2em] {
&F_2\\
G_1^k&R\\};
\path[->] (m-1-2) edge  node[above]{$q_1^k$} (m-2-1);
\path[->] (m-1-2) edge  node[right]{${D_0^k}'$} (m-2-2);
\path[->] (m-2-1) edge  node[above]{$\delta_1^k$} (m-2-2);
\end{tikzpicture}
\quad\quad
\begin{tikzpicture}[baseline=(current  bounding  box.center)]
 \matrix (m) [matrix of math nodes,row sep=3em,column sep=4em,minimum width=2em] {
F_3&F_2\\
G_2^k&G_1^k\\};
\path[->] (m-1-1) edge  node[above]{$D_3$} (m-1-2);
\path[->] (m-1-1) edge  node[right]{$q_2^k$} (m-2-1);
\path[->] (m-2-1) edge  node[above]{$\delta_2^k$} (m-2-2);
\path[->] (m-1-2) edge  node[right]{$q_1^k$} (m-2-2);
\end{tikzpicture}
\]
the result follows from \cite[Theorem 2.6]{DGTrimmed}. First we check the commutativity left-hand diagram:
\begin{align*}
{D_0^k}'(f_i) &= T_{k,i},\\
\delta_1^k(q_1^k(f_i)) &= \delta_1^k\left(\sum_{l=1}^3 c_{i,k,l}u_l^k\right) = \sum_{l=1}^3c_{i,k,l}z_l = T_{k,i}.
\end{align*}
Now we check the commutativity of the right-hand diagram:

\begin{align*}
q_1^k(D_3(g)) &= q_1^k\left( \sum_{i=1}^m y_if_i \right) = \sum_{i=1}^m \sum_{l=1}^3 y_ic_{i,k,l}u_l^k=\left(\sum_{i=1}^my_ic_{i,k,1}\right)u_1^k+\left(\sum_{i=1}^my_ic_{i,k,2}\right)u_2^k+\left(\sum_{i=1}^my_ic_{i,k,3}\right)u_3^k,\end{align*}
\begin{samepage}
\begin{align*}
\delta_2^k(q_2^k(g)) &= \delta_2^k(d^k_{1,2}v^k_{1,2} + d^k_{1,3}v^k_{1,3} + d^k_{2,3}v^k_{2,3})\\
&= d^k_{1,2}(z_1u^k_2-z_2u^k_1) + d^k_{1,3}(z_1u^k_3-z_3u^k_1) + d^k_{2,3}(z_2u^k_3-z_3u^k_2)\\
&= (-z_2d^k_{1,2} - z_3d^k_{1,3})u^k_1 + (z_1d^k_{1,2} - z_3d^k_{2,3})u^k_2 + (z_1d^k_{1,3} + z_2d^k_{2,3})u^k_3.
\end{align*}
\end{samepage}

We check that $\sum_{i=1}^my_ic_{i,k,1}=-z_2d_{1,2}^k-z_3d_{1,3}^k$; the remaining two equalities are similarly checked.

It follows from \eqref{eq:singlepf} and the definition of $\sigma_{i,k,r}$ that for every $k\in\{1,\ldots,m\}\backslash\{i\}$ the following equality holds
\begin{equation}\label{eq:y_i}
y_i=\sum_{r=1}^m\sigma_{i,k,r}T_{k,r}\pf{i,k,r}{T}.
\end{equation}

Expanding $\sum_{i=1}^my_ic_{i,k,1}$ using \eqref{eq:y_i} and \eqref{eq:Tji} gives
\begin{align*}
\sum_{i=1}^m \sum_{r=1}^m \sigma_{i,k,r} c_{r,k,1}c_{i,k,1} \pf{i,k,r}{T} z_1 + \sum_{i=1}^m \sum_{r=1}^m \sigma_{i,k,r} c_{r,k,2}c_{i,k,1} \pf{i,k,r}{T} z_2 + \sum_{i=1}^m \sum_{r=1}^m \sigma_{i,k,r} c_{r,k,3}c_{i,k,1} \pf{i,k,r}{T} z_3.
\end{align*}

We notice that the coefficient of $z_1$ is 0 since, by \eqref{eq:Sig3Rel},
\[
\sigma_{i,k,r}c_{r,k,1}c_{i,k,1}\pf{i,k,r}{T}=-\sigma_{i,r,k}c_{i,k,1}c_{r,k,1}\pf{i,r,k}{T},
\]
and that the coefficients of $z_2$ and $z_3$ are $-d_{1,2}^k$ and $-d_{1,3}^k$ respectively.
\end{proof}

\section{Products on the resolution}
We start by defining several constants that will be used in the statement of the main theorem of this section. Let $\alpha,\beta\in\{1,2,3\}$ with $\alpha<\beta$, let $k$ be an integer between $1$ and $t$, and let $i,j$ be integers between $1$ and $m$. 

We first define the constants $d_{\alpha,\beta}^{k,i,j}$ which will be used in products \textit{a} and \textit{e} of the theorem.
\[
d^{k,i,j}_{\alpha,\beta} \colonequals \sum_{r=1}^{m} \sigma_{i,j,r} \sum_{h=1}^m \sigma_{i,j,r,h,k} \pf{i,j,r,h,k}{T}c_{r,k,\beta}c_{h,k,\alpha}.
\]

Let $l\in\{1,2,3\}$. The next set of constants will be used in product $b$.

\begin{align*}
    d_{\alpha,\beta}^{k,i,j,l}&\colonequals z_ld_{\alpha,\beta}^{k,i,j},\quad\mathrm{for}\;k\neq i,\\
    d^{i,i,j,l}_{\alpha,\beta}&\colonequals 0\quad\mathrm{if}\;\{\alpha,\beta,l\}=\{1,2,3\},\\
    d^{i,i,j,\alpha}_{\alpha,\beta}&\colonequals\sum_{r=1}^m\sigma_{i,j,r}\pf{i,j,r}{T}c_{r,i,\beta},\\
    d^{i,i,j,\beta}_{\alpha,\beta}&\colonequals-\sum_{r=1}^m\sigma_{i,j,r}\pf{i,j,r}{T}c_{r,i,\alpha}.
\end{align*}

We set $d_{\beta,\alpha}=-d_{\alpha,\beta}$ for any number of superscripts.

The DG algebra structure on $G_\bullet^k$ (with $1\leq k\leq t$) is the usual exterior product of the Koszul complex, namely
\[
u_1^ku_2^k=v_{1,2}^k,\quad u_1^ku_3^k=v_{1,3}^k,\quad u_2^ku_3^k=v_{2,3}^k,\quad u_1^ku_2^ku_3^k=w^k,
\]
and the usual skew-commutativity rules. A product in the DG algebra $G_\bullet^k$ will be denoted as $-\cdot_{G^k}-$. In the statement (and proof) of the theorem we will denote a product in the DG algebra $F_\bullet$ as $-\cdot_F-$. We will denote products on the resolution given in \Cref{thm:res} by an unadorned $-\cdot-$.

\begin{remark}
By \cite[Lemma 4.9]{DGTrimmed}, if the ideal we trim is contained in the square of the maximal ideal, then the only products on the resolution of the trimmed ideal, defined in \cite[Theorem 3.2]{DGTrimmed}, that may induce nontrivial products in homology, are the products $F_1'\cdot F_1'$ and $F_1'\cdot F_2$. In the next Theorem, we explicitly construct these products. Moreover, we also construct the products of the form $F_1'\cdot G_1^i$, since they are used in the computation of the products of the form $F_1'\cdot F_2$. The interested reader can find the complete DG algebra structure on this resolution in the Appendix. We label the products in the next Theorem so that they are consistent with the labels in \Cref{thm:Appendix}.
\end{remark}

\begin{theorem}\label{thm:prod}
A DG algebra structure on the resolution constructed in \Cref{thm:res} is given by the following product rules

\noindent \underline{a. $F_1'\otimes F_1' \rightarrow F_2 \oplus (\oplus_{k=1}^t G_2^k)$}\\
\begin{equation*}
e_i\cdot e_j\colonequals e_i\cdot_Fe_j+\sum_{k=1}^t d^{k,i,j}_{1,2}v^k_{1,2}+d^{k,i,j}_{1,3}v^k_{1,3} + d^{k,i,j}_{2,3}v^k_{2,3},
\end{equation*}
for $t+1\leq i,j\leq m$.

\vspace{4mm}
\noindent\underline{b. $F_1'\otimes G_1^i \rightarrow F_2 \oplus (\oplus_{k=1}^t G_2^k)$}\\
\begin{equation*}
e_j\cdot u_l^i\colonequals z_le_i\cdot_Fe_j +\sum_{k=1}^t d^{k,i,j,l}_{1,2}v^k_{1,2}+d^{k,i,j,l}_{1,3}v^k_{1,3} + d^{k,i,j,l}_{2,3}v^k_{2,3},
\end{equation*}
for $1\leq i\leq t$ and $t+1\leq j\leq m$.

\vspace{4mm}
\noindent \underline{e. $F_1' \otimes F_2 \rightarrow F_3 \oplus \left( \oplus_{k=1}^t G_3^k \right)$}\\
\begin{equation*}
e_i \cdot f_j\colonequals \begin{cases}
e_i\cdot_Ff_j, & t+1\leq j\leq m \\
- \sum_{r=1}^m c_{r,j,3} d^{j,i,r}_{1,2} w^j, & 1\leq j\leq t
\end{cases}
\end{equation*}
for $t+1\leq i \leq m$.
\end{theorem}
Before proving the theorem we provide an example
\begin{example}\label{ex:DGARes}
In this example we compute the products $a,b,e$ on the resolution constructed in \Cref{ex:FreeRes}. For product $a$, we need the constants $d_{\alpha,\beta}^{k,i,j}$
\begin{align*}
d_{1,2}^{1,i,j}&=\sum_{r=1}^5\sum_{h=1}^5\pf{i,j,r,h,1}{T}c_{r,1,2}c_{h,1,1}=0,\\
d_{1,3}^{1,i,j}&=\sum_{r=1}^5\sum_{h=1}^5\pf{i,j,r,h,1}{T}c_{r,1,3}c_{h,1,1}=\pf{i,j,5,4,1}{T},\\
d_{2,3}^{1,i,j}&=\sum_{r=1}^5\sum_{h=1}^5\pf{i,j,r,h,1}{T}c_{r,1,3}c_{h,1,2}=0.
\end{align*}
So, the products of type $a$ are
\begin{center}
\def\arraystretch{1.25}
\begin{tabular}{|c||c|c|c|c|c|}
\hline
& $e_2$ & $e_3$ & $e_4$ & $e_5$\\
\hline\hline
$e_2$ & $0$ & $zf_4+xf_5+v_{1,3}^1$ & $zf_3$ & $yf_1+xf_3$\\
\hline
$e_3$ & $-$ & $0$ & $yf_1+zf_2$ & $zf_1+xf_2$\\
\hline
$e_4$ & $-$ & $-$ & $0$ & $xf_1$\\
\hline
$e_5$ & $-$ & $-$ & $-$ & $0$\\
\hline
\end{tabular}
\end{center}
where the $-$ indicates products that can be deduced from the skew-commutativity of the product. For product $b$, we need the constants $d_{\alpha,\beta}^{k,i,j,l}$
\begin{align*}
d_{1,2}^{1,1,j,3}&=d_{1,3}^{1,1,j,2}=d_{2,3}^{1,1,j,1}=0,\\
d_{1,2}^{1,1,j,1}&=d_{2,3}^{1,1,j,3}=\sum_{r=1}^5\pf{1,j,r}{T}c_{r,1,2}=0,\\
d_{1,2}^{1,1,j,2}&=d_{1,3}^{1,1,j,3}=\sum_{r=1}^5\pf{1,j,r}{T}c_{r,1,1}=\pf{1,j,4}{T},\\
d_{1,3}^{1,1,j,1}&=d_{2,3}^{1,1,j,2}=\sum_{r=1}^5\pf{1,j,r}{T}c_{r,1,3}=\pf{1,j,5}{T}.
\end{align*}
So, the products of type $b$ are
\begin{center}
\def\arraystretch{1.25}
\begin{tabular}{|c||c|c|c|c|c|}
\hline
& $u_1^1$ & $u_2^1$ & $u_3^1$ \\
\hline\hline
$e_2$ & $xyf_5+yv_{1,3}^1$ & $y^2f_5+yv_{2,3}^1$ & $yzf_5$ \\
\hline
$e_3$ & $xyf_4+xzf_5+zv_{1,3}^1$ & $y^2f_4+yzf_5+yv_{1,2}^1+zv_{2,3}^1$ & $yzf_4+z^2f_5+yv_{1,3}^1$ \\
\hline
$e_4$ & $xyf_3+x^2f_5+xv_{1,3}^1$ & $y^2f_3+xyf_5+xv_{2,3}^1$ & $yzf_3+xzf_5$ \\
\hline
$e_5$ & $xyf_2+xzf_3+x^2f_4$ & $y^2f_2+yzf_3+xyf_4+xv_{1,2}^1$ & $yzf_2+z^2f_3+xzf_4+xv_{1,3}^1$ \\
\hline
\end{tabular}
\end{center}
While the products of type $e$ are
\begin{center}
\def\arraystretch{1.25}
\begin{tabular}{|c||c|c|c|c|c|}
\hline
& $f_1$ & $f_2$ & $f_3$ & $f_4$ & $f_5$\\
\hline\hline
$e_2$ & $0$ & $g$ & $0$ & $0$ & $0$\\
\hline
$e_3$ & $0$ & $0$ & $g$ & $0$ & $0$\\
\hline
$e_4$ & $0$ & $0$ & $0$ & $g$ & $0$\\
\hline
$e_5$ & $0$ & $0$ & $0$ & $0$ & $g$\\
\hline
\end{tabular}
\end{center}
\end{example}

\begin{proof}[Proof of \Cref{thm:prod}]\hphantom{a}\\\vspace{-0.1in}

\noindent \underline{$a$. $F_1'\otimes F_1' \rightarrow F_2 \oplus (\oplus_{k=1}^t G_2^k)$}\\
By \cite[Theorem 3.2]{DGTrimmed}, it suffices to show that
\[
\delta_2^k(d^{k,i,j}_{1,2}v^k_{1,2}+d^{k,i,j}_{1,3}v^k_{1,3} + d^{k,i,j}_{2,3}v^k_{2,3})=q_1^k(e_i\cdot_Fe_j).
\]
Indeed, the left side is equal to
\[
(-d^{k,i,j}_{1,2}z_2 - d^{k,i,j}_{1,3}z_3)u^k_1 + (d^{k,i,j}_{1,2}z_1 -d^{k,i,j}_{2,3}z_3)u^k_2 + (d^{k,i,j}_{1,3}z_1+d^{k,i,j}_{2,3}z_2)u^k_3,
\]
while the right side is equal to
\begin{align*}
    q_1^k(e_i\cdot_Fe_j) 
    &= q_1^k\left(\sum_{r=1}^m \sigma_{i,j,r} \pf{i,j,r}{T} f_r\right)\\
    &= \sum_{r=1}^{m} \sigma_{i,j,r} \pf{i,j,r}{T} (c_{r,k,1}u^k_1 + c_{r,k,2}u^k_2 + c_{r,k,3}u^k_3)\\
    &= \sum_{r=1}^{m} \sigma_{i,j,r} \sum_{h=1}^m \sigma_{i,j,r,h,k} T_{k,h} \pf{i,j,r,h,k}{T} (c_{r,k,1}u^k_1 + c_{r,k,2}u^k_2 + c_{r,k,3}u^k_3)\\
    &= \sum_{r=1}^{m} \sigma_{i,j,r} \sum_{h=1}^m \sigma_{i,j,r,h,k} (c_{h,k,1}z_1+c_{h,k,2}z_3+c_{h,k,1}z_3) \pf{i,j,r,h,k}{T} (c_{r,k,1}u^k_1 + c_{r,k,2}u^k_2 + c_{r,k,3}u^k_3),
\end{align*}
where the first equality comes from the product in $F_\bullet$, the second equality comes from applying $q_1^k$ as defined in \Cref{thm:res}, the third equality comes from \eqref{eq:triplepf}, and the fourth equality comes from \eqref{eq:Tji}. We show that the coefficients of $u_1^k$ are the same; a similar argument can be used for the coefficients of $u_2^k$ and $u_3^k$. We need to show that the following equality holds
\begin{align*}
    -d^{k,i,j}_{1,2}z_2 - d^{k,i,j}_{1,3}z_3 &= \sum_{r=1}^{m} \sigma_{i,j,r} \sum_{h=1}^m \sigma_{i,j,r,h,k} \pf{i,j,r,h,k}{T}c_{r,k,1}c_{h,k,1}z_1\\
    &+ \sum_{r=1}^{m} \sigma_{i,j,r} \sum_{h=1}^m \sigma_{i,j,r,h,k} \pf{i,j,r,h,k}{T}c_{r,k,1}c_{h,k,2}z_2\\ &+ \sum_{r=1}^{m} \sigma_{i,j,r} \sum_{h=1}^m \sigma_{i,j,r,h,k} \pf{i,j,r,h,k}{T}c_{r,k,1}c_{h,k,3}z_3.
\end{align*}
\noindent
The coefficient of $z_1$ on the right hand side of the display above is zero since
\[
\sigma_{i,j,r}\sigma_{i,j,r,h,k}=-\sigma_{i,j,h}\sigma_{i,j,h,r,k},
\]
which can be verified using \eqref{eq:Sigma3Theta} and \eqref{eq:Sigma5Theta}. The coefficients of $z_2$ and $z_3$ match for the same reason.

\vspace{4mm}
\noindent\underline{$b$. $F_1'\otimes G_1^i \rightarrow F_2 \oplus (\oplus_{k=1}^t G_2^k)$}\\
By \cite[Theorem 3.2]{DGTrimmed}, it suffices to show that
\begin{align}
\delta_2^i(d^{i,i,j,l}_{1,2}v^i_{1,2}+d^{i,i,j,l}_{1,3}v^i_{1,3} + d^{i,i,j,l}_{2,3}v^i_{2,3}) &= y_ju_l^i+z_lq_1^i(e_i\cdot_Fe_j),\label{eq:bi}\\
\delta_2^k(d^{k,i,j,l}_{1,2}v^k_{1,2}+d^{k,i,j,l}_{1,3}v^k_{1,3} + d^{k,i,j,l}_{2,3}v^k_{2,3})&=z_lq_1^k(e_i\cdot_Fe_j),\quad\mathrm{for}\;k\neq i.\label{eq:bk}
\end{align}
We first prove \eqref{eq:bi}. The left side is equal to
\[
(-d^{i,i,j,l}_{1,2}z_2 - d^{i,i,j,l}_{1,3}z_3)u^i_1 + (d^{i,i,j,l}_{1,2}z_1 -d^{i,i,j,l}_{2,3}z_3)u^i_2 + (d^{i,i,j,l}_{1,3}z_1+d^{i,i,j,l}_{2,3}z_2)u^i_3,
\]
while the right side is equal to
\begin{align*}
    &y_ju_l^i + z_lq_1^i(e_i\cdot_Fe_j) \\
    = &y_ju_l^i + z_lq_1^i\left( \sum_{r=1}^{m} \sigma_{i,j,r} \pf{i,j,r}{T} f_r \right)\\
    = & y_ju_l^i + z_l \sum_{r=1}^{m} \sigma_{i,j,r} \pf{i,j,r}{T} (c_{r,i,1}u_1^i + c_{r,i,2}u_2^i + c_{r,i,3}u_3^i), 
\end{align*}
where the first equality comes from the product in $F_\bullet$ and the second equality comes from applying $q_1^k$ as defined in \Cref{thm:res}. To prove that the coefficients of $u_1^i,u_2^i$ and $u_3^i$ match, we assume that $l=1$; the remaining cases are proved similarly. The coefficients of $u_2^i$ and $u_3^i$ match by the definition of the constants $d^{i,i,j,1}_{1,2}$, $d^{i,i,j,1}_{1,3}$ and $d^{i,i,j,1}_{2,3}$. Expanding $y_j$ using \eqref{eq:singlepf} and \eqref{eq:Sig3Rel}, it follows that the coefficient of $u_1^i$ in $y_ju_1^i+z_1q_1^i(e_1\cdot_Fe_j)$ is
\begin{align*}
     &-z_1\sum_{r=1}^m \sigma_{i,j,r} \pf{i,j,r}{T} c_{r,i,1} -z_2\sum_{r=1}^m \sigma_{i,j,r} \pf{i,j,r}{T} c_{r,i,2}\\ 
     &-z_3\sum_{r=1}^m \sigma_{i,j,r} \pf{i,j,r}{T} c_{r,i,3} 
     +z_1\sum_{r=1}^{m} \sigma_{i,j,r} \pf{i,j,r}{T} c_{r,i,1}\\
     = &-z_2\sum_{r=1}^m \sigma_{i,j,r} \pf{i,j,r}{T} c_{r,i,2} -z_3\sum_{r=1}^m \sigma_{i,j,r} \pf{i,j,r}{T} c_{r,i,3}\\
     = &-d_{1,2}^{i,i,j,1}z_2 -d_{1,3}^{i,i,j,1}z_3.
\end{align*}
Now we prove \eqref{eq:bk}. By the definitions of $d_{1,2}^{k,i,j,l}$, $d_{1,3}^{k,i,j,l}$ and $d_{2,3}^{k,i,j,l}$ when $k\neq i$, the left hand side is equal to
\begin{align*}
\delta_2^k(z_ld^{k,i,j}_{1,2}v^k_{1,2}+z_ld^{k,i,j}_{1,3}v^k_{1,3} + z_ld^{k,i,j}_{2,3}v^k_{2,3}) = z_l\delta_2^k(d^{k,i,j}_{1,2}v^k_{1,2}+d^{k,i,j}_{1,3}v^k_{1,3} + d^{k,i,j}_{2,3}v^k_{2,3}) = z_lq_1^k(e_i\cdot_Fe_j),
\end{align*}
where the last equality was shown in the proof of product \textit{a}.

\vspace{4mm}
\noindent \underline{$e$. $F_1' \otimes F_2 \rightarrow F_3 \oplus \left( \oplus_{k=1}^t G_3^k \right)$}\\
Set
\begin{align*}
    g_2^k &\colonequals \left(\sum_{r=t+1}^{m} T_{r,j} d_{1,2}^{k,i,r}\right) v^k_{1,2} + \left(\sum_{r=t+1}^{m} T_{r,j} d_{1,3}^{k,i,r}\right) v^k_{1,3} + \left(\sum_{r=t+1}^{m} T_{r,j} d_{2,3}^{k,i,r}\right) v^k_{2,3},\quad\mathrm{with}\;1\leq k\leq t,\\
    g_2^{k,r} &\colonequals \left(\sum_{l=1}^3 c_{j,k,l} d_{1,2}^{r,k,i,l}\right)v^r_{1,2} + \left(\sum_{l=1}^3 c_{j,k,l} d_{1,3}^{r,k,i,l}\right) v^r_{1,3} + \left(\sum_{l=1}^3 c_{j,k,l}d_{2,3}^{r,k,i,l}\right) v^r_{2,3},\quad\mathrm{with}\;1\leq k,r\leq t.
\end{align*}
By \cite[Theorem 3.2]{DGTrimmed}, it suffices to show that
\begin{align}
e_i\cdot \left( d_2'(f_j) - \sum_{k=1}^t q_1^k(f_j)\right) &= e_i\cdot_F d_2'(f_j) + \sum_{k=1}^t g_2^k - \sum_{k=1}^t \left( \delta_1^k\circ q_1^k(f_j) e_k\cdot_F e_i + \sum_{r=1}^t g_2^{k,r} \right),\label{eq:e1}\\
q_2^k(e_i\cdot_F f_j) + g_2^k - \sum_{r=1}^t g_2^{r,k} &= 0,\;\mathrm{with}\;k\neq j,\label{eq:e2}\\
q_2^j(e_i\cdot_F f_j) + g_2^j - \sum_{r=1}^t g_2^{r,j} &= \delta_3^j\left(\sum_{r=1}^m c_{r,j,3} d^{j,i,r}_{1,2} w^j\right).\label{eq:e3}
\end{align}
We first prove \eqref{eq:e1}. The left hand side is equal to 
\begin{align*}
&e_i\cdot\sum_{r=t+1}^m T_{r,j} e_r - e_i\cdot\sum_{k=1}^t \sum_{l=1}^3 c_{j,k,l}u_l^k\\
= &\sum_{r=t+1}^m T_{r,j} e_i\cdot_Fe_r + \sum_{r=t+1}^{m} T_{r,j} \sum_{k=1}^t d_{1,2}^{k,i,r}v^k_{1,2} + d_{1,3}^{k,i,r}v^k_{1,3} + d_{2,3}^{k,i,r}v^k_{2,3}\\
&- \sum_{k=1}^t \sum_{l=1}^3 c_{j,k,l} z_l e_k\cdot_Fe_i - \sum_{k=1}^t \sum_{r=1}^t \sum_{l=1}^3 c_{j,k,l} d_{1,2}^{r,k,i,l}v^r_{1,2} + d_{1,3}^{r,k,i,l}v^r_{1,3} + d_{2,3}^{r,k,i,l}v^r_{2,3},
\end{align*}
while the right hand side is equal to 
\begin{align*}
&e_i\cdot_F\sum_{r=t+1}^m T_{r,j} e_r + \sum_{k=1}^t \sum_{r=t+1}^{m} T_{r,j} \left(d_{1,2}^{k,i,r}v^k_{1,2} + d_{1,3}^{k,i,r}v^k_{1,3} + d_{2,3}^{k,i,r}v^k_{2,3}\right)\\
&- \sum_{k=1}^t  \delta_1^k\circ q_1^k(f_j) e_k\cdot_F e_i - \sum_{k=1}^t \sum_{r=1}^t \left(\sum_{l=1}^3 c_{j,k,l} d_{1,2}^{r,k,i,l}v^r_{1,2} + d_{1,3}^{r,k,i,l}v^r_{1,3} + d_{2,3}^{r,k,i,l}v^r_{2,3}\right),
\end{align*}
which coincides with the left hand side. Now we prove \eqref{eq:e2}. We show that the coefficient of $v_{1,2}^k$ on the left hand side of \eqref{eq:e2} is zero; the remaining cases are similarly checked. The coefficient of $v_{1,2}^k$ is
\begin{align*}
    &\delta_{i,j}\sum_{r=1}^m\sum_{h=1}^m \sigma_{r,k,h} c_{r,k,2} c_{h,k,1} \pf{r,k,h}{T} + \sum_{r=t+1}^{m} T_{r,j} d_{1,2}^{k,i,r} - \sum_{r=1}^t \sum_{l=1}^3 c_{j,r,l} d_{1,2}^{k,r,i,l}\\
    = &\delta_{i,j}\sum_{r=1}^m\sum_{h=1}^m \sigma_{r,k,h} c_{r,k,2} c_{h,k,1} \pf{r,k,h}{T} + \sum_{r=t+1}^{m} T_{r,j} d_{1,2}^{k,i,r} - \sum_{\underset{r\neq k}{r=1}}^t \sum_{l=1}^3 c_{j,r,l} d_{1,2}^{k,r,i,l} - \sum_{l=1}^3 c_{j,k,l} d_{1,2}^{k,k,i,l}\\
    = &\delta_{i,j}\sum_{r=1}^m\sum_{h=1}^m \sigma_{r,k,h} c_{r,k,2} c_{h,k,1} \pf{r,k,h}{T} + \sum_{r=t+1}^{m} T_{r,j} d_{1,2}^{k,i,r} - \sum_{\underset{r\neq k}{r=1}}^t \sum_{l=1}^3 z_lc_{j,r,l} d_{1,2}^{k,r,i} - \sum_{l=1}^3 c_{j,k,l} d_{1,2}^{k,k,i,l}\\
    = &\delta_{i,j}\sum_{r=1}^m\sum_{h=1}^m \sigma_{r,k,h} c_{r,k,2} c_{h,k,1} \pf{r,k,h}{T} + \sum_{r=t+1}^{m} T_{r,j} d_{1,2}^{k,i,r} - \sum_{\underset{r\neq k}{r=1}}^t T_{r,j} d_{1,2}^{k,r,i} - \sum_{l=1}^3 c_{j,k,l} d_{1,2}^{k,k,i,l}\\
    = &\delta_{i,j}\sum_{r=1}^m\sum_{h=1}^m \sigma_{r,k,h} c_{r,k,2} c_{h,k,1} \pf{r,k,h}{T} + \sum_{r=t+1}^{m} T_{r,j} d_{1,2}^{k,i,r} + \sum_{\underset{r\neq k}{r=1}}^t T_{r,j} d_{1,2}^{k,i,r} - \sum_{l=1}^3 c_{j,k,l} d_{1,2}^{k,k,i,l}\\
    = &\delta_{i,j}\sum_{r=1}^m\sum_{h=1}^m \sigma_{r,k,h} c_{r,k,2} c_{h,k,1} \pf{r,k,h}{T} + \sum_{\underset{r\neq k}{r=1}}^m T_{r,j} d_{1,2}^{k,i,r} - \sum_{l=1}^3 c_{j,k,l} d_{1,2}^{k,k,i,l}\\
    = &\delta_{i,j}\sum_{r=1}^m\sum_{h=1}^m \sigma_{r,k,h} c_{r,k,2} c_{h,k,1} \pf{r,k,h}{T} + \sum_{\underset{r\neq k}{r=1}}^m T_{r,j} d_{1,2}^{k,i,r} - c_{j,k,1} d_{1,2}^{k,k,i,1} - c_{j,k,2} d_{1,2}^{k,k,i,2}\\
    = &\delta_{i,j}\sum_{r=1}^m\sum_{h=1}^m \sigma_{r,k,h} c_{r,k,2} c_{h,k,1} \pf{r,k,h}{T} + \sum_{\underset{r\neq k}{r=1}}^m T_{r,j} d_{1,2}^{k,i,r}\\
    &- \sum_{r=1}^m \sigma_{k,i,r} \pf{k,i,r}{T} c_{j,k,1}c_{r,k,2} + \sum_{r=1}^m \sigma_{k,i,r} \pf{k,i,r}{T} c_{r,k,1}c_{j,k,2}\\
    = &\delta_{i,j}\sum_{r=1}^m\sum_{h=1}^m \sigma_{r,k,h} c_{r,k,2} c_{h,k,1} \pf{r,k,h}{T} + \sum_{\underset{r\neq k}{r=1}}^m T_{r,j} d_{1,2}^{k,i,r} + \sum_{r=1}^m \sigma_{k,i,r} \pf{k,i,r}{T} (c_{r,k,1}c_{j,k,2} - c_{j,k,1}c_{r,k,2})
\end{align*}
where the first equality follows from separating the case $r=k$, the second equality follows from the definition of $d_{1,2}^{k,r,i,l}$, the third equality follows from \eqref{eq:Tji}, the fourth equality follows from definition of $d_{1,2}^{k,i,r}$, the fifth equality follows from combining summations, the sixth and seventh equalities follow from definition of $d_{1,2}^{k,k,i,l}$, and the eighth equality follows from combining summations. First, we consider the case $i\neq j$, in which the coefficient of $v_{1,2}^k$ becomes
\begin{align*}
    &\sum_{\underset{r\neq k}{r=1}}^m T_{r,j} d_{1,2}^{k,i,r} + \sum_{r=1}^m \sigma_{k,i,r} \pf{k,i,r}{T} (c_{r,k,1}c_{j,k,2} - c_{j,k,1}c_{r,k,2})\\
    = &\sum_{\underset{r\neq k}{r=1}}^m T_{r,j} \sum_{h=1}^m \sigma_{i,r,h} \sum_{s=1}^m \sigma_{i,r,h,s,k} \pf{i,r,h,s,k}{T} c_{s,k,1} c_{h,k,2} + \sum_{h=1}^m\sigma_{k,i,h}\pf{k,i,h}{T} (c_{h,k,1}c_{j,k,2} - c_{j,k,1}c_{h,k,2})\\
    = &\sum_{\underset{r\neq k}{r=1}}^m T_{r,j} \sum_{\underset{h\neq j}{h=1}}^m \sigma_{i,r,h} \sum_{s=1}^m \sigma_{i,r,h,s,k} \pf{i,r,h,s,k}{T} c_{s,k,1} c_{h,k,2} + \sum_{\underset{r\neq k}{r=1}}^m T_{r,j} \sigma_{i,r,j} \sum_{h=1}^m \sigma_{i,r,j,h,k} \pf{i,r,j,h,k}{T} c_{h,k,1} c_{j,k,2}\\ 
    &+ \sum_{\underset{h\neq j}{h=1}}^m \sigma_{k,i,h}\sum_{\underset{r\neq k}{r=1}}^m T_{j,r} \sigma_{k,i,h,r,j} \pf{k,i,h,r,j}{T} (c_{h,k,1}c_{j,k,2} - c_{j,k,1}c_{h,k,2})\\
    = &\sum_{\underset{r\neq k}{r=1}}^m T_{r,j} \sum_{\underset{h\neq j}{h=1}}^m \sigma_{i,r,h} \sum_{s=1}^m \sigma_{i,r,h,s,k} \pf{i,r,h,s,k}{T} c_{s,k,1} c_{h,k,2} + \sum_{\underset{r\neq k}{r=1}}^m T_{r,j} \sigma_{i,r,j} \sum_{\underset{h\neq j}{h=1}}^m \sigma_{i,r,j,h,k} \pf{i,r,j,h,k}{T} c_{h,k,1} c_{j,k,2}\\ 
    &- \sum_{\underset{h\neq j}{h=1}}^m \sigma_{k,i,h}\sum_{\underset{r\neq k}{r=1}}^m T_{r,j} \sigma_{k,i,h,r,j} \pf{k,i,h,r,j}{T} (c_{h,k,1}c_{j,k,2} - c_{j,k,1}c_{h,k,2})\\
    = &\sum_{\underset{r\neq k}{r=1}}^m T_{r,j} \sum_{\underset{h\neq j}{h=1}}^m \sigma_{i,r,h} \sum_{s=1}^m \sigma_{i,r,h,s,k} \pf{i,r,h,s,k}{T} c_{s,k,1} c_{h,k,2} + \sum_{\underset{r\neq k}{r=1}}^m T_{r,j} \sum_{\underset{h\neq j}{h=1}}^m \sigma_{k,i,h}\sigma_{k,i,h,r,j} \pf{k,i,h,r,j}{T}c_{j,k,1}c_{h,k,2}\\
    = &\sum_{\underset{r\neq k}{r=1}}^m T_{r,j} \sum_{\underset{h\neq j}{h=1}}^m \sigma_{i,r,h} \sum_{\underset{s\neq j}{s=1}}^m \sigma_{i,r,h,s,k} \pf{i,r,h,s,k}{T} c_{s,k,1} c_{h,k,2} + \sum_{\underset{r\neq k}{r=1}}^m T_{r,j} \sum_{\underset{h\neq j}{h=1}}^m \sigma_{i,r,h} \sigma_{i,r,h,j,k} \pf{i,r,h,j,k}{T} c_{j,k,1} c_{h,k,2}\\ 
    &+ \sum_{\underset{r\neq k}{r=1}}^m T_{r,j} \sum_{\underset{h\neq j}{h=1}}^m \sigma_{k,i,h}\sigma_{k,i,h,r,j} \pf{k,i,h,r,j}{T}c_{j,k,1}c_{h,k,2}\\
    = &\sum_{\underset{r\neq k}{r=1}}^m T_{r,j} \sum_{\underset{h\neq j}{h=1}}^m \sigma_{i,r,h} \sum_{\underset{s\neq j}{s=1}}^m \sigma_{i,r,h,s,k} \pf{i,r,h,s,k}{T} c_{s,k,1} c_{h,k,2}\\
    = &\sum_{\underset{s\neq j}{s=1}}^m c_{s,k,1}\sum_{\underset{h\neq j}{h=1}}^m c_{h,k,2} \sum_{\underset{r\neq k}{r=1}}^m \sigma_{i,r,h} \sigma_{i,r,h,s,k} T_{r,j} \pf{i,r,h,s,k}{T}\\
    = &0,
\end{align*}
where the first equality follows from the definition of $d_{1,2}^{k,i,r}$ and reindexing, the second equality follows from separating the case $h=j$, reindexing and using \eqref{eq:triplepf}, the third equality changes $T_{j,r}$ in the third term with $-T_{r,j}$, the fourth equality follows from cancellation via \eqref{eq:Sigma3Theta} and \eqref{eq:Sigma5Theta}, the fifth equality follows from separating the case $s=j$, the sixth equality follows from cancellation via \eqref{eq:Sigma3Theta} and \eqref{eq:Sigma5Theta}, the seventh equality follows from rearranging the summation, and the eighth equality follows from \eqref{eq:5pf0}.

We move to consider the case $i=j$, which forces $k\neq j$, and the coefficient of $v_{1,2}^k$ is
\begin{align*}
&\sum_{r=1}^m\sum_{h=1}^m \sigma_{r,k,h} c_{r,k,2} c_{h,k,1} \pf{r,k,h}{T} + \sum_{\underset{r\neq k}{r=1}}^m T_{r,i} d_{1,2}^{k,i,r} + \sum_{r=1}^m \sigma_{k,i,r} \pf{k,i,r}{T} (c_{r,k,1}c_{i,k,2} - c_{i,k,1}c_{r,k,2})\\
= &\sum_{r=1}^m\sum_{h=1}^m \sigma_{r,k,h} c_{r,k,2} c_{h,k,1} \pf{r,k,h}{T} + \sum_{r=1}^m T_{r,i} \sum_{s=1}^m \sigma_{i,r,s} \sum_{h=1}^m \sigma_{i,r,s,h,k} \pf{i,r,s,h,k}{T} c_{s,k,2}c_{h,k,1}\\
&+ \sum_{r=1}^m \sigma_{k,i,r} \pf{k,i,r}{T} (c_{r,k,1}c_{i,k,2} - c_{i,k,1}c_{r,k,2})\\
= &\sum_{\underset{r\neq i}{r=1}}^m\sum_{h=1}^m \sigma_{r,k,h} c_{r,k,2} c_{h,k,1} \pf{r,k,h}{T} + \sum_{h=1}^m \sigma_{i,k,h} c_{i,k,2} c_{h,k,1} \pf{i,k,h}{T}\\ 
&+ \sum_{r=1}^m T_{r,i} \sum_{s=1}^m \sigma_{i,r,s} \sum_{h=1}^m \sigma_{i,r,s,h,k} \pf{i,r,s,h,k}{T} c_{s,k,2}c_{h,k,1} + \sum_{h=1}^m \sigma_{k,i,h} \pf{k,i,h}{T} (c_{h,k,1}c_{i,k,2} - c_{i,k,1}c_{h,k,2})\\
= &\sum_{\underset{r\neq i}{r=1}}^m\sum_{h=1}^m \sigma_{r,k,h} c_{r,k,2} c_{h,k,1} \pf{r,k,h}{T}\\
&+ \sum_{r=1}^m T_{r,i} \sum_{s=1}^m \sigma_{i,r,s} \sum_{h=1}^m \sigma_{i,r,s,h,k} \pf{i,r,s,h,k}{T} c_{s,k,2}c_{h,k,1} - \sum_{h=1}^m\sigma_{k,i,h}\pf{k,i,h}{T}c_{h,k,2}c_{i,k,1}\\
= &\sum_{\underset{r\neq i}{r=1}}^m\sum_{\underset{h\neq i}{h=1}}^m \sigma_{r,k,h} c_{r,k,2} c_{h,k,1} \pf{r,k,h}{T} + \sum_{r=1}^m \sigma_{r,k,i} c_{r,k,2} c_{i,k,1} \pf{r,k,i}{T}\\
&+ \sum_{r=1}^m T_{r,i} \sum_{s=1}^m \sigma_{i,r,s} \sum_{h=1}^m \sigma_{i,r,s,h,k} \pf{i,r,s,h,k}{T} c_{s,k,2}c_{h,k,1} -  \sum_{r=1}^m\sigma_{k,i,r}\pf{k,i,r}{T}c_{r,k,2}c_{i,k,1}\\
= &\sum_{\underset{r\neq i}{r=1}}^m\sum_{\underset{h\neq i}{h=1}}^m \sigma_{r,k,h} c_{r,k,2} c_{h,k,1} \pf{r,k,h}{T} + \sum_{r=1}^m T_{r,i} \sum_{s=1}^m \sigma_{i,r,s} \sum_{h=1}^m \sigma_{i,r,s,h,k} \pf{i,r,s,h,k}{T} c_{s,k,2}c_{h,k,1}\\
= &\sum_{\underset{r\neq i}{r=1}}^m\sum_{\underset{h\neq i}{h=1}}^m \sigma_{r,k,h} c_{r,k,2} c_{h,k,1} \sum_{\underset{s\neq i}{s=1}}^m \sigma_{r,k,h,s,i}T_{i,s}\pf{r,k,h,s,i}{T} + \sum_{r=1}^m T_{r,i} \sum_{s=1}^m \sigma_{i,r,s} \sum_{h=1}^m \sigma_{i,r,s,h,k} \pf{i,r,s,h,k}{T} c_{s,k,2}c_{h,k,1}\\
= & \sum_{\underset{s\neq i}{s=1}}^m\sum_{\underset{h\neq i}{h=1}}^m \sigma_{s,k,h} c_{s,k,2} c_{h,k,1} \sum_{\underset{r\neq i}{r=1}}^m \sigma_{s,k,h,r,i}T_{i,r}\pf{s,k,h,r,i}{T}
+ \sum_{r=1}^m T_{r,i} \sum_{s=1}^m \sigma_{i,r,s} \sum_{h=1}^m \sigma_{i,r,s,h,k} \pf{i,r,s,h,k}{T} c_{s,k,2}c_{h,k,1}\\
= & 
\sum_{\underset{r\neq i}{r=1}}^m T_{i,r}
\sum_{\underset{s\neq i}{s=1}}^m 
\sum_{\underset{h\neq i}{h=1}}^m
\sigma_{s,k,h} \sigma_{s,k,h,r,i}\pf{s,k,h,r,i}{T} c_{s,k,2} c_{h,k,1}
- \sum_{r=1}^m T_{i,r} \sum_{s=1}^m \sum_{h=1}^m \sigma_{i,r,s} \sigma_{i,r,s,h,k} \pf{i,r,s,h,k}{T} c_{s,k,2}c_{h,k,1}\\
= &0
\end{align*}
\vspace{-0.02in}\noindent
where the first equality follows from the definition of $d_{1,2}^{k,i,r}$ , the second equality follows from separating the case $r=i$, the third equality follows from cancellation via \eqref{eq:Sig3Rel}, the fourth equality follows from separating the case $h=i$, the fifth equality follows from cancellation via \eqref{eq:Sig3Rel}, the sixth equality follows from \eqref{eq:triplepf}, the seventh equality follows from reindexing, the eighth equality changes $T_{r,i}$ in the second term with $-T_{i,r}$, and the ninth equality follows from cancellation via \eqref{eq:Sigma3Theta} and \eqref{eq:Sigma5Theta}. Notice that we can take $r,s,h\neq i$ in the second term of the eighth equality as the pfaffian will be zero otherwise. 

Now we show \eqref{eq:e3}. We show that the coefficient of $v_{1,2}^j$ is the same on both sides of \eqref{eq:e3}; the remaining cases are similarly checked. The coefficient of $v_{1,2}^j$ on the left hand side of \eqref{eq:e3} is computed as in \eqref{eq:e2}, and is equal to
\begin{align*}
&\delta_{i,j}\sum_{r=1}^m\sum_{h=1}^m \sigma_{r,j,h} c_{r,j,2} c_{h,j,1} \pf{r,j,h}{T} + \sum_{\underset{r\neq j}{r=1}}^m T_{r,j} d_{1,2}^{j,i,r} + \sum_{r=1}^m \sigma_{j,i,r} \pf{j,i,r}{T} (c_{r,j,1}c_{j,j,2} - c_{j,j,1}c_{r,j,2})\\
= & \delta_{i,j}\sum_{r=1}^m\sum_{h=1}^m \sigma_{r,j,h} c_{r,j,2} c_{h,j,1} \pf{r,j,h}{T} + \sum_{\underset{r\neq j}{r=1}}^m T_{r,j} d_{1,2}^{j,i,r}.
\end{align*}

Notice that $i\neq j$, since in \eqref{eq:e3}, the element $g_2^j$ is only defined for $1\leq j \leq t$ and we have $t+1\leq i \leq m$. Hence the coefficient above is computed as in \eqref{eq:e2} and is equal to
\begin{align*}
    &\sum_{\underset{s\neq j}{s=1}}^m c_{s,j,1} \sum_{\underset{h\neq j}{h=1}}^m c_{h,j,2} \sum_{\underset{r\neq j}{r=1}}^m \sigma_{i,r,h} \sigma_{i,r,h,s,j} T_{r,j} \pf{i,r,h,s,j}{T}\\
    = &\sum_{\underset{s\neq j}{s=1}}^m c_{s,j,1} \sum_{\underset{h\neq j}{h=1}}^m c_{h,j,2} \sum_{\underset{r\neq j}{r=1}}^m \sigma_{i,r,h} \sigma_{i,r,h,s,j}  \pf{i,r,h,s,j}{T} \sum_{l=1}^3 c_{j,r,l}z_l\\
    = &z_1\sum_{\underset{h\neq j}{h=1}}^m c_{h,j,2} \left(\sum_{\underset{s\neq j}{s=1}}^m  \sum_{\underset{r\neq j}{r=1}}^m c_{s,j,1} c_{j,r,1}\sigma_{i,r,s} \sigma_{i,r,s,h,j} \pf{i,r,s,h,j}{T} \right)\\
    &+ z_2\sum_{\underset{s\neq j}{s=1}}^m c_{s,j,1} \left(\sum_{\underset{h\neq j}{h=1}}^m \sum_{\underset{r\neq j}{r=1}}^m c_{h,j,2} c_{j,r,2}\sigma_{i,r,s} \sigma_{i,r,s,h,j} \pf{i,r,s,h,j}{T} \right)\\
    &+z_3\left(\sum_{\underset{s\neq j}{s=1}}^m c_{s,j,1} \sum_{\underset{h\neq j}{h=1}}^m c_{h,j,2} \sum_{\underset{r\neq j}{r=1}}^m c_{j,r,3}\sigma_{i,r,s} \sigma_{i,r,s,h,j} \pf{i,r,s,h,j}{T} \right)\\
    = &z_1\left(\sum_{\underset{s\neq j}{s=1}}^m c_{s,j,1} 0 \right) + z_2\left(\sum_{\underset{h\neq j}{h=1}}^m c_{h,j,2} 0 \right) + z_3\left(\sum_{\underset{s\neq j}{s=1}}^m c_{s,j,1} \sum_{\underset{h\neq j}{h=1}}^m c_{h,j,2} \sum_{\underset{r\neq j}{r=1}}^m c_{j,r,3}\sigma_{i,r,s} \sigma_{i,r,s,h,j} \pf{i,r,s,h,j}{T} \right)\\
    = &\left(\sum_{\underset{s\neq j}{s=1}}^m c_{s,j,1} \sum_{\underset{h\neq j}{h=1}}^m c_{h,j,2} \sum_{\underset{r\neq j}{r=1}}^m c_{j,r,3}\sigma_{i,r,s} \sigma_{i,r,s,h,j} \pf{i,r,s,h,j}{T} \right) z_3\\
    = &\left(\sum_{\underset{r\neq j}{r=1}}^m c_{r,j,3} \sum_{\underset{h\neq j}{h=1}}^m \sigma_{i,r,h} \sum_{\underset{s\neq j}{s=1}}^m \sigma_{i,r,h,s,j} \pf{i,r,h,s,j}{T} c_{s,j,1}c_{h,j,2} \right)z_3\\
    = &\left(\sum_{\underset{r\neq j}{r=1}} c_{r,j,3} d^{j,i,r}_{1,2} \right)z_3
\end{align*}
where the first equality follows from \eqref{eq:Tji}, the second equality follows from rearranging the summations, the third equality follows from \eqref{eq:Sigma3Theta} and \eqref{eq:Sigma5Theta}, the fourth equality follows from cancellation, the fifth equality follows from rearranging the summation, and the sixth equality follows from definition of $d_{1,2}^{j,i,r}$. 
Hence the nonzero coefficients of $v_{1,2}^j$, $v_{1,3}^j$, and $v_{2,3}^j$ in \eqref{eq:e3} are, respectively, 
\[
\left(\sum_{\underset{r\neq j}{r=1}} c_{r,j,3} d^{j,i,r}_{1,2}\right)z_3, \quad \left(-\sum_{\underset{h\neq j}{h=1}} c_{h,j,2} d^{j,i,h}_{1,3}\right)z_2, \quad \mathrm{ and\;} \left(\sum_{\underset{s\neq j}{s=1}} c_{s,j,1} d^{j,i,s}_{2,3}\right)z_1.
\]
One can check that the quantities in the parentheses in the display above are the same.
\end{proof}

\section{Tor Algebras}

In this section we study the Tor algebra of ideals obtained by trimming the pfaffian generators of a Gorenstein ideal of grade 3. As a consequence of the main result of this section, we show that the conjectures in \cite[7.4 Conjectures]{Linkage} on ideals of class $\mathbf{G}(r)$ hold true in our context. Moreover, it will also follow that $r$ is always given by the lower bound provided in \cite[(2.4) Theorem]{trimming}.

\begin{notation}\label{notation}
Let $T$ be a skew-symmetric matrix of odd size $m$ with entries in $\m$, and let $t$ be an integer between 1 and $m$. Recall by \Cref{thm:res} that $Q_1$ is the matrix 
\[
Q_1=\begin{pmatrix}q_1^1\\\vdots\\ q_1^t\end{pmatrix}=\begin{pmatrix}c_{1,1,1}&c_{2,1,1}&\cdots&c_{t,1,1}\\
c_{1,1,2}&c_{2,1,2}&\cdots&c_{t,1,2}\\
c_{1,1,3}&c_{2,1,3}&\cdots&c_{t,1,3}\\
\vdots&\vdots&\vdots&\vdots\\
c_{1,t,1}&c_{2,t,1}&\cdots&c_{t,t,1}\\
c_{1,t,2}&c_{2,t,2}&\cdots&c_{t,t,2}\\
c_{1,t,3}&c_{2,t,3}&\cdots&c_{t,t,3}\\
\end{pmatrix},
\]
where the entries are defined by \eqref{eq:Tji}. The matrices appearing in the displays of \Cref{prop:5G} and \Cref{lem:Cond1&2} are submatrices of the block $q_1^k\otimes\kk$ for $k=1,\ldots,t$, in which the entries are denoted with a bar to indicate the residue class modulo $\m$. We denote by $p(T,t)$ the number of pivot columns of $Q_1\otimes_R\kk$ among the last $m-t$ columns. Moreover, we denote by $(C_\bullet,\partial_\bullet)$ the resolution given in \Cref{thm:res}.
\end{notation}

\begin{convention}\label{convention}
To make the results of this section easier to read, we will denote the classes $\mathbf{H}(0,0)$ and $\mathbf{H}(0,1)$ by $\mathbf{G}(0)$ and $\mathbf{G}(1)$, respectively. In particular, we say that ideals of class $\mathbf{H}(0,0)$ and $\mathbf{H}(0,1)$ are of class $\mathbf{G}$.
\end{convention}

\begin{lemma}\label{lem:Cond1&2}
Let $T$ be a skew-symmetric matrix of size $5$ with entries in $\m$ and let $t$ be an integer between 1 and $5$. Consider the following two conditions:
\begin{enumerate}
\item For every $i,j,k$ distinct with $t+1\leq i,j\leq 5$ and $1\leq k\leq t$, set $\{r,h\}=[5]\backslash\{k,i,j\}$. The $2\times 2$ minors of the matrix
\[
\begin{pmatrix}\ov{c_{h,k,1}}&\ov{c_{h,k,2}}&\ov{c_{h,k,3}}\\
\ov{c_{r,k,1}}&\ov{c_{r,k,2}}&\ov{c_{r,k,3}}
\end{pmatrix}
\]
are zero.
\item For every $i\in[5]$ and $k\in[5]$ with $1\leq k\leq t$ and $t+1\leq i\leq 5$, set $\{h,s,r\}=[5]\backslash\{i,k\}$. The determinant of the matrix
\[
\begin{pmatrix}\ov{c_{h,k,1}}&\ov{c_{s,k,1}}&\ov{c_{r,k,1}}\\
\ov{c_{h,k,2}}&\ov{c_{s,k,2}}&\ov{c_{r,k,2}}\\
\ov{c_{h,k,3}}&\ov{c_{s,k,3}}&\ov{c_{r,k,3}}
\end{pmatrix}
\]
is zero.
\end{enumerate}
If $t\leq 3$, then condition (1) implies condition (2).
\end{lemma}

\begin{proof}
Since $t\leq 3$, we may choose $s\in[5]\backslash{\{i,k\}}$ such that $t+1\leq s\leq 5$. Expanding the determinant of the matrix in \emph{(2)} along the second column immediately shows that \emph{(1)} implies \emph{(2)}.
\end{proof}

\begin{lemma}\label{lem:onlyG}
Let $I$ be a Gorenstein ideal of grade 3 generated by $((-1)^{i+1}\pf{i}{T})_{i=1,\ldots,5}$ for some skew-symmetric matrix $T$ of size $5$ with entries in $\m$. Let $t$ be an integer between 1 and $5$. Let $J$ be the ideal obtained by trimming the first $t$ generators of $I$. Then $\ov{C_1}\ov{C_2}\subseteq\ov{F_3}$ if and only if condition (2) in \Cref{lem:Cond1&2} is satisfied, where $C_1,C_2$ are defined in \Cref{notation}, and $F_3$ is defined in \Cref{ch:GorRes}.
\end{lemma}

\begin{proof}
In product $e$, for $\ov{C_1}\ov{C_2}\subseteq\ov{F_3}$, we must have $\sum_{r=1}^m \ov{c_{r,k,3} d^{k,i,r}_{1,2}}=0$ for all $t+1\leq i\leq m$ and $1\leq k\leq t$ with $k\in[5]$. Fixing $\{r,h,s\}=[5]\backslash\{i,k\}$, we can expand this summation as 
\begin{align*}
&\ov{c_{r,k,3}}\left( \sigma_{i,r,h}\sigma_{i,r,h,s,k} \ov{c_{s,k,1}}\ov{c_{h,k,2}} + \sigma_{i,r,s}\sigma_{i,r,s,h,k} \ov{c_{h,k,1}}\ov{c_{s,k,2}} \right)\\
&+ \ov{c_{h,k,3}}\left( \sigma_{i,h,r}\sigma_{i,h,r,s,k} \ov{c_{s,k,1}}\ov{c_{r,k,2}} + \sigma_{i,h,s}\sigma_{i,h,s,r,k} \ov{c_{r,k,1}}\ov{c_{s,k,2}} \right)\\
&+ \ov{c_{s,k,3}}\left( \sigma_{i,s,r}\sigma_{i,s,r,h,k} \ov{c_{h,k,1}}\ov{c_{r,k,2}} + \sigma_{i,s,h}\sigma_{i,s,h,r,k} \ov{c_{r,k,1}}\ov{c_{h,k,2}} \right)\\
= &\sigma_{i,r,h}\sigma_{i,r,h,s,k}\ov{c_{r,k,3}}\left(\ov{c_{s,k,1}}\ov{c_{h,k,2}} - \ov{c_{h,k,1}}\ov{c_{s,k,2}} \right)\\
&+ \sigma_{i,h,r}\sigma_{i,h,r,s,k}\ov{c_{h,k,3}}\left(\ov{c_{s,k,1}}\ov{c_{r,k,2}} - \ov{c_{r,k,1}}\ov{c_{s,k,2}} \right)\\
&+ \sigma_{i,s,r}\sigma_{i,s,r,h,k}c_{s,k,3}\left(\ov{c_{h,k,1}}\ov{c_{r,k,2}} - \ov{c_{r,k,1}}\ov{c_{h,k,2}} \right)\\
= &-\sigma_{i,r,h}\sigma_{i,r,h,s,k}(-\ov{c_{r,k,3}}(\ov{c_{s,k,1}}\ov{c_{h,k,2}} - \ov{c_{h,k,1}}\ov{c_{s,k,2}} ) + \ov{c_{h,k,3}}(\ov{c_{s,k,1}}\ov{c_{r,k,2}} - \ov{c_{r,k,1}}\ov{c_{s,k,2}} ) \\&+ \ov{c_{s,k,3}}(\ov{c_{h,k,1}}\ov{c_{r,k,2}} - \ov{c_{r,k,1}}\ov{c_{h,k,2}} ) )\\
= &-\sigma_{i,r,h}\sigma_{i,r,h,s,k}(\ov{c_{h,k,3}}(\ov{c_{s,k,1}}\ov{c_{r,k,2}} - \ov{c_{r,k,1}}\ov{c_{s,k,2}} ) - \ov{c_{r,k,3}}(\ov{c_{s,k,1}}\ov{c_{h,k,2}} - \ov{c_{h,k,1}}\ov{c_{s,k,2}} ) \\&+ \ov{c_{s,k,3}}(\ov{c_{h,k,1}}\ov{c_{r,k,2}} - \ov{c_{r,k,1}}\ov{c_{h,k,2}} ) )\\
= &-\sigma_{i,r,h}\sigma_{i,r,h,s,k}\;\mathrm{det}
\begin{pmatrix}\ov{c_{h,k,1}}&\ov{c_{s,k,1}}&\ov{c_{r,k,1}}\\
\ov{c_{h,k,2}}&\ov{c_{s,k,2}}&\ov{c_{r,k,2}}\\
\ov{c_{h,k,3}}&\ov{c_{s,k,3}}&\ov{c_{r,k,3}}
\end{pmatrix},
\end{align*}
which shows that condition \emph{(2)} in \Cref{lem:Cond1&2} is equivalent to $\ov{C_1}\ov{C_2}\subseteq\ov{F_3}$.
\end{proof}

\begin{notation} Let $C_\bullet$ be the resolution constructed in \Cref{thm:res}. We denote by $\ee_1,\ldots,\ee_m$ the basis of $C_\bullet\otimes_R\kk$ induced by $e_1,\ldots, e_m$. A similar notation is enforced for the remaining basis elements of $C_\bullet$.
\end{notation}

\begin{lemma}\label{lem:ChangeOfBasis}
Let $I$ be a Gorenstein ideal of grade 3 generated by $((-1)^{i+1}\pf{i}{T})_{i=1,\ldots,m}$ for some skew-symmetric matrix $T$ of odd size and with entries in $\m$. Let $t$ be an integer between 1 and $m$. Let $J$ be the ideal obtained by trimming the first $t$ generators of $I$. Let $C_\bullet$ be the resolution of $R/J$ constructed in \Cref{thm:res}. If the only nonzero products on $C_\bullet\otimes_R\kk$ are $\ee_i\ff_i$ for $i\geq t+1$, then the ideal $J$ is of class $\mathbf{G}(m-t-p(T,t))$.
\end{lemma}
\begin{proof}
Since row operations on $Q_1\otimes_R\kk$ change basis elements that do not contribute to the product, we can assume that $Q_1\otimes_R\kk$ is in row reduced echelon form. One can perform elementary column operations on $Q_1\otimes_R\kk$ so that every pivot is the only nonzero element in its row, this will allow to split off the nonminimal part of $C_\bullet\otimes_R\kk$. These column operations yield a new basis of the form
\[
\ff_j'=\ff_j+\sum_{l\in P\backslash\{j\}}\alpha_{j,l}\ff_l,\quad\alpha_{j,l}\in\kk\;\forall j,l,
\]
where $P$ is the index set of the pivot columns of $Q_1\otimes_R\kk$.\newline
Set
\[
\ee_i'=\begin{cases}
\displaystyle\ee_i-\sum_{k\not\in P}\alpha_{k,i}\ee_k\quad\mathrm{if}\;i\in P\\
\ee_i\hphantom{-\sum_{k\not\in P}\alpha_{k,i}}\;\quad\mathrm{if}\;i\not\in P.
\end{cases}
\]
We first show that $\ee_1',\ldots,\ee_m'$ is a basis. Let $b_1,\ldots, b_m\in\kk$ such that
\[
b_1\ee_1'+\cdots b_m\ee_m'=0,
\]
then
\[
\sum_{i\in P}b_i\left(\ee_i-\sum_{k\not\in P}\alpha_{k,i}\ee_k\right)+\sum_{i\not\in P}b_i\ee_i=0.
\]
If $j\in P$, then the coefficient of $\ee_j$ in the previous expression is $b_j$, therefore $b_j=0$ for $j\in P$. If $j\not\in P$, then the coefficient of $\ee_j$ is 
\[
b_j-\sum_{i\in P}b_i\alpha_{j,i}=b_j,
\]
since $b_i=0$ for $i\in P$. Therefore $b_j=0$ for $j\not\in P$ as well. This shows that $\ee_1',\ldots,\ee_m'$ is a basis.\newline
Now we study the products between the elements $\{\ee_i'\}_{i\in[m]}$ and $\{\ff_j'\}_{j\not\in P}$. Let $j\not\in P$. If $j\geq t+1$, then $\ee_j'\ff_j'=\ee_j\ff_j'=\gg$. If $j\leq t$, then $\ee_j'\ff_j'=\ee_j\ff_j'=0$. If $i\in P$, then
\[
\ee_i'\ff_j'=\alpha_{j,i}\gg-\alpha_{j,i}\gg=0.
\]
If $i\not\in P$ and $i\neq j$, then $\ee_i'\ff_j'=\ee_i\ff_j'=0$.\newline
Since a basis of $\Tor_2^R(R,J,\kk)$ is given by $\{\ff_j'\}_{j\not\in P}$, it follows that the ideal $J$ is of class $\mathbf{G}$. Since there are only $m-t-p(T,t)$ of the elements $\ff_j'$ contributing to a nonzero product, it follows that the class is $\mathbf{G}(m-t-p(T,t))$.
\end{proof}

\begin{proposition}\label{prop:5G}
Let $I$ be a Gorenstein ideal of grade 3 generated by $((-1)^{i+1}\pf{i}{T})_{i=1,\ldots,5}$ for some skew-symmetric matrix $T$ of size $5$ with entries in $\m$. Let $t$ be an integer between 1 and $5$. Let $J$ be the ideal obtained by trimming the first $t$ generators of $I$. Then $J$ is of class $\mathbf{G}$ if and only if
the $2\times 2$ minors of the matrix
\[
\begin{pmatrix}\ov{c_{h,k,1}}&\ov{c_{h,k,2}}&\ov{c_{h,k,3}}\\
\ov{c_{r,k,1}}&\ov{c_{r,k,2}}&\ov{c_{r,k,3}}
\end{pmatrix}
\]
are zero for every $i,j,k$ distinct with $t+1\leq i,j\leq 5$ and $1\leq k\leq t$, $\{h,r\}=[5]\backslash\{k,i,j\}$.
\end{proposition}

\begin{proof}
For $t=4,5$ the condition on the minors is vacuously satisfied, therefore we need to show that in this case $J$ is always of class $\mathbf{G}$.

If $t=5$, then $J$ is clearly of class $\mathbf{G}(0)$ since there are no nonzero products in the Tor algebra.

Now we assume $t=4$. By (the proof of) \Cref{lem:onlyG}, the coefficients appearing in product $e$ of $C_\bullet$ for $1\leq j\leq4$ are given by determinants.
If one of these determinants is nonzero, then there are three pivot columns of $Q_1\otimes_R\kk$ among the first four columns.  The skew-symmetry of $T$ forces the remaining fourth column to be a pivot column. If $p(T,4)=1$, then all columns of $Q_1\otimes_R\kk$ are pivot columns and therefore there are no nonzero products in the Tor algebra, hence $J$ is of class $\mathbf{G}(0)$. If $p(T,4)=0$, then the only nonzero product in the Tor algebra is the one induced by $e_5f_5$, and therefore $J$ is of class $\mathbf{G}(1)$.

If the determinants appearing in product $e$ of $C_\bullet$ for $1\leq j\leq4$ are all zero, then, by \Cref{lem:ChangeOfBasis}, $J$ is of class $\mathbf{G}(1-p(T,4))$.

Before moving on to the case $1\leq t\leq3$, we make the following observation. In product $a$ the coefficient of $v^k_{\alpha,\beta}$ is
\[
\sigma_{i,j,r} \sigma_{i,j,r,h,k} \ov{c_{h,k,\alpha}c_{r,k,\beta}} + \sigma_{i,j,h} \sigma_{i,j,h,r,k} \ov{c_{r,k,\alpha}c_{h,k,\beta}},
\]
where $h$ and $r$ are defined as in the statement of the proposition. Using \eqref{eq:Sigma3Theta} and \eqref{eq:Sigma5Theta}, the previous display is equal to
\[
\sigma_{i,j,r} \sigma_{i,j,r,h,k} \left(\ov{c_{h,k,\alpha}c_{r,k,\beta}} - \ov{c_{r,k,\alpha}c_{h,k,\beta}}\right),
\]
which is, up to a sign, one of the $2\times2$ minors of the matrices in the statement of the Proposition.

Now we assume $1\leq t\leq3$. If $J$ is of class $\mathbf{G}$, then products of type $a$ must be zero and therefore the $2\times2$ minors of the matrices in the statement of the Proposition must be zero by the previous observation.

If the minors are zero, then, by the previous observation, the products of type $a$ are also zero. Moreover, by by \Cref{lem:Cond1&2} and \Cref{lem:onlyG}, it follows that $\ov{C_1}\ov{C_2}\subseteq\ov{F_3}$. By \Cref{lem:ChangeOfBasis}, $J$ is of class $\mathbf{G}(5-t-p(T,t))$.
\end{proof}

We are finally ready to prove the main result of the paper.
\begin{theorem}\label{thm:TrimClass}
Let $I$ be a Gorenstein ideal of grade 3 generated by $((-1)^{i+1}\pf{i}{T})_{i=1,\ldots,m}$ for some skew-symmetric matrix $T$ of odd size $m\geq 5$ with entries in $\m$. Let $t$ be an integer between 1 and $m$. Let $J$ be the ideal obtained by trimming the first $t$ generators of $I$. Then $J$ is an ideal of format
\[
(1, m+2t-\rank(Q_1\otimes_R\kk), m+3t-\rank(Q_1\otimes_R\kk), 1+t).
\]
Moreover, 
\begin{enumerate}
\item If $m=5$, then $J$ is of class $\mathbf{G}$ if and only if condition (1) in \Cref{lem:Cond1&2} is satisfied. 
\item If $m\geq7$, then $J$ is of class $\mathbf{G}$.
\end{enumerate}
Furthermore, if $J$ is of class $\mathbf{G}(r)$, then $r=m-t-p(T,t)$, where $Q_1$ and $p(T,t)$ are defined in \Cref{notation}.
\end{theorem}

\begin{proof}
We first show that the format of $J$ is the one given by the formula above. A resolution of $R/J$ is given in \Cref{thm:res}. The ranks of the free modules in this resolution are $1$, $m+2t$, $m+3t$, and $1+t$ for degree $0$, $1$, $2$, and $3$, respectively. The only differential that may contain units is $\partial_2$. The only submatrix of $\partial_2$ that may contain units is $-Q_1$,
therefore when passing to a minimal resolution, one needs to remove exactly $\rank(Q_1\otimes_R\kk)$ copies of $R$ from the domain and codomain of $\partial_2$.

If $m=5$, then \Cref{prop:5G} shows that $J$ is of class $\mathbf{G}$ if and only if condition \emph{(1)} in \Cref{lem:Cond1&2} is satisfied. The proof of \Cref{prop:5G} shows that in this case $r=5-t-p(T,t)$. If $m\geq 7$, then the only products of the ones listed in \Cref{thm:prod} that are nonzero modulo $\m$ are the products $\ee_i\ff_i$ with $i\geq t+1$. Using \Cref{convention}, \Cref{lem:ChangeOfBasis} and the product tables in \cref{chk:ProdTables}, this forces $J$ to be of class $\mathbf{G}(m-t-p(T,t))$.
\end{proof}

\begin{remark}
We point out that, adopting the notation and hypotheses of \Cref{thm:TrimClass}, if the entries of $T$ are in $\m^2$, then the ideal $J$ is always of class $\mathbf{G}(m-t)$. This also follows from \cite[Corollary 4.12]{DGTrimmed}.
\end{remark}

\begin{remark}\label{rem:conj}
With \Cref{convention} in place, the conjectures given by Christensen, Veliche and Weyman in \cite[7.4 Conjectures]{Linkage}, together with \cite[(2.4) Theorem]{trimming}
reduce to the two conditions of \Cref{cor:conj} for ideals obtained by trimming the pfaffian generators of Gorenstein ideals. In the following corollary $\mu(J)$ denotes the minimal number of generators of $J$.


\end{remark}

\begin{corollary}\label{cor:conj}
Let $I$ be a Gorenstein ideal of grade 3 generated by $((-1)^{i+1}\pf{i}{T})_{i=1,\ldots,m}$ for some skew-symmetric matrix $T$ of odd size with entries in $\m$. Let $J$ be the ideal obtained by trimming the first $t$ generators of $I$. If $J$ is of class $\mathbf{G}(r)$, then 
\begin{enumerate}
\item if $t=1$, then $r=\mu(J)-3$.
\item if $t\geq2$, then $r\leq \mu(J)-4$.
\end{enumerate}
\end{corollary}

\begin{proof}
By \Cref{thm:TrimClass}, we have $\mu(J)=m+2t-\rank(Q_1\otimes_R\kk)$ and $r=m-t-p(T,t)$. \begin{enumerate}
\item If $t=1$, since $T_{1,1}=0$, it follows by definition that $\rank(Q_1\otimes_R\kk)=p(T,t)$. Therefore,
\[
r=m-1-p(T,1)=(m+2-p(T,1))-3=\mu(J)-3.
\]
\item By definition it follows that
\[
\rank(Q_1\otimes_R\kk)-p(T,t)\leq t.
\]
If $t\geq2$, then $t\leq 3t-4$, therefore
\[
\rank(Q_1\otimes_R\kk)-p(T,t)\leq 3t-4,
\]
which is equivalent to
\[
m-t-p(T,t)\leq m+2t-\rank(Q_1\otimes_R\kk)-4;
\]
hence $r\leq \mu(J)-4$.\qedhere
\end{enumerate}
\end{proof}

\begin{remark}\label{rem:perm}
We point out that in \Cref{cor:conj} one does not need to necessarily trim the first $t$ generators of $I$, since one can always reduce to this case by conjugating the matrix $T$ by a permutation matrix. In particular, if one wishes to trim generators $i_1,i_2,\dots,i_t$, one conjugates $T$ by the matrix corresponding to the permutation
$\begin{psmallmatrix}
i_1&i_2&\cdots&i_t\\
1&2&\cdots&t
\end{psmallmatrix}$.
Indeed, it is an easy linear algebra exercise to show that if $E$ is an elementary matrix corresponding to a swap operation, then the sequence of principal minors of the matrix $EAE$ is obtained by applying the same swap operation to the sequence of principal minors of $A$. Therefore, if $A$ is skew-symmetric then the sequence of pfaffians of $EAE$ is obtained by applying the same swap operation to the sequence of sub-maximal pfaffians of $A$, up to a sign. Now one uses that a permutation matrix is a product of elementary matrices corresponding to swap operations. 

If one wishes to understand the format and class of an ideal obtained by trimming \emph{any} $t$ pfaffian generators of a Gorenstein ideal of grade 3, then by the observation above one can reduce to the case of trimming the first $t$ generators and apply \Cref{thm:TrimClass}.
\end{remark}

\section{Realizability}

In this section we address the realizability question, i.e. for what integers $\ell,n,r$ can one find an ideal of format $(1,\ell,\ell+n-1,n)$ and of class $\mathbf{G}(r)$? If the ideal is obtained from a Gorenstein ideal of grade 3 that is not a complete intersection by trimming $t$ pfaffian generators, then $n=t+1$. Moreover, it follows from \Cref{thm:TrimClass} and from the observation that 
\[
0\leq\mathrm{rank}(Q_1\otimes_R\kk)-p(T,t)\leq t,
\]
that $2t\leq \ell-r\leq 3t$. In \cite[Corollary 5.9]{DGTrimmed}, the realizability of ideals with $\ell-r\in\{3t-3,3t-2,3t\}$ was already shown. In this section we show that $\ell-r=3t-1$ is not possible, i.e. ideals with $\ell-r=3(n-1)-1=3n-4$ do not arise via trimming the pfaffian generators of a Gorenstein ideal, and construct a family of examples realizing $\ell-r=2t$.

\def\rddots#1{\cdot^{\cdot^{\cdot^{#1}}}}

\begin{proposition}\label{prop:Gap}
Let $J$ be an ideal obtained by trimming $t$ pfaffian generators of a Gorenstein ideal of grade 3 that is not a complete intersection. If $J$ is minimally generated by $\ell$ elements and is of class $\mathbf{G}(r)$, then $\ell-r\neq 3t-1$.
\end{proposition}

\begin{proof}
Without loss of generality, we can assume that $J$ is obtained by trimming the first $t$ pfaffian generators of a Gorenstein ideal of grade 3, see \Cref{rem:perm}. Seeking a contradiction, we assume that $\ell-r=3t-1$. By \Cref{thm:TrimClass}, it follows that $$\mathrm{rank}(Q_1\otimes_R\kk)-p(T,t)=1,$$ i.e. the matrix $Q_1\otimes_R\kk$ has exactly one pivot column among the first $t$ columns. We argue that this is not possible. Recall that $Q_1$ is a block matrix with $t$ blocks, see \Cref{notation}. Let $i$ be the pivot column among the first $t$ columns. Let $j$ be the the block containing the first nonzero entry in column $i$. By the skew-symmetry of the matrix presenting the Gorenstein ideal, it follows that $i\neq j$ and that there is a nonzero entry in column $j$ and block $i$ with zeroes to the left of it. This forces column $j$ to be a pivot column, which is a contradiction.
\end{proof}

\begin{remark}
A computer search made with the \verb+TorAlgebra+  package of Macaulay2 showed that ideals in three dimensional polynomial rings of format $(1,\ell,\ell+n-1,n)$ and of class $\mathbf{G}(r)$ with $\ell-r=3n-4$ are abundant. For example one can take the ideal $J=(xz^2, xy^2-z^3, x^2y, x^3, z^4, y^3z, y^5)$ in $\mathbb{Q}[x,y,z]$ which has format $(1,7,9,3)$ and class $\mathbf{G}(2)$. In particular, it follows by \Cref{prop:Gap} that these ideals cannot arise by trimming the pfaffian generators of a Gorenstein ideal.
\end{remark}

\begin{chunk}\label{chk:Ex1}
Set $R=\kk[[x,y,z]]$ and let $s$ be a positive integer. We define the $s\times s$ matrix $U_s$ as the matrix
\[
U_s\colonequals\begin{pmatrix}
&&&x&z\\
&&\iddots&z&y^2\\
&\iddots&\iddots&\iddots\\
x&z&\iddots\\
z&y^2
\end{pmatrix},
\]
where the empty entries are 0. Let $V_s$ be the $(2s+1)\times (2s+1)$ matrix given by
\[
V_s\colonequals\left(\arraycolsep=1.4pt\def\arraystretch{1.5}
\begin{array}{cc|cc|ccccc}
O && O_x && && U_s &&\\
\hline
-(O_x)^T && 0 && && {}^{y^2}O &&\\
\hline
-U_s && -(^{y^2}O)^T && && O &&
\end{array}\right)
\]
where $O_x$ is an appropriately sized column matrix with an $x$ at the bottom and zeroes elsewhere and ${}^{y^2}O$ is an appropriately sized row matrix with $y^2$ in the leftmost entry and zeroes elsewhere. Let $T_s$ be the $(4s+3) \times (4s+3)$ matrix given by 
\[
T_s\colonequals\left(\arraycolsep=1.4pt\def\arraystretch{1.5}
\begin{array}{cc|cc|ccccc}
V_s && O_x && && U_{2s+1} &&\\
\hline
-(O_x)^T && 0 && && {}^{y^2}O &&\\
\hline
-U_{2s+1} && -(^{y^2}O)^T && && O &&
\end{array}\right).
\]
Let $I_s$ be the ideal $((-1)^{i+1}\pf{i}{T_s}))_{i=1,\ldots,4s+3}$, and let $J_s$ be the ideal obtained by trimming the first $2s+1$ generators of $I_s$.
\end{chunk}

\begin{lemma}\label{lem:OddEx}
The ideals $I_s$ are Gorenstein ideals of grade 3.
\end{lemma}

\begin{proof}
By \cite[Theorem 2.1]{BuchEisen} it suffices to show that these ideals contain a regular sequence of length 3. We claim that $\pf{1}{T_s}$, $\pf{2s+1}{T_s}$ and $\pf{4s+3}{T_s}$ form a regular sequence. 

The matrix obtained by removing the $1^{\mathrm{st}}$ row and $1^{\mathrm{st}}$ column from $T_s$ has determinant $y^{8s+4}$, it follows that $\pf{1}{T_s}=\pm y^{4s+2}$. 

By \cite[Theorem 2]{BlockDet}, the determinant of the matrix obtained by removing the $2s+1^{\mathrm{st}}$ row and $2s+1^{\mathrm{st}}$ column from $T_s$ is $(\det U_{2s+1})^2$, and therefore $\pf{2s+1}{T_s}=\pm\det U_{2s+1}$, which contains $z^{2s+1}$ as a summand by a proof similar to \cite[(3.2) Lemma]{trimming}. 

By taking the cofactor expansion of the determinant of the matrix obtained from $T_s$ by removing the $4s+3^\mathrm{rd}$ row and column, it is clear that this determinant has $\pm x^{4s+2}$ as a summand. It follows that $\pf{4s+1}{T_s}$ has $\pm x^{2s+1}$ as a summand.

\end{proof}

\begin{proposition}\label{prop:OddEx}
The ideal $J_s$ has format $(1,6s+2,8s+3,2s+2)$ and is of class $\mathbf{G}(2s)$.
\end{proposition}
\begin{proof}
An elementary analysis of the matrix $Q_1\otimes_R\kk$ shows that the $x$'s and $-x$'s in the block $V_s$ correspond to $2s$ pivot columns among the first $2s+1$ columns of $Q_1\otimes_R\kk$. The only other pivot among the first $2s+1$ columns corresponds to the $z$ in the first row of $V_s$. The remaining pivots of $Q_1\otimes_R\kk$ correspond to the $z$ in the $s+1^{\mathrm{st}}$ row of $T_s$ and to the $x$ in the $2s+1^{\mathrm{st}}$ row of $T_s$. Therefore $\mathrm{rank}(Q_1\otimes_R\kk)=2s+3$ and $p(T_s,2s+1)=2$. The result now follows from \Cref{thm:TrimClass}.
\end{proof}

\begin{chunk}
Set $R=\kk[[x,y,z]]$ and let $s$ be an integer larger than or equal to 2. The $s\times s$ matrix $U_s$ is defined as in \Cref{chk:Ex1}.
Let $V_s'$ be the $(2s)\times (2s)$ matrix given by
\[
V_s'\colonequals\left(\arraycolsep=1.4pt\def\arraystretch{1.5}
\begin{array}{cc|cccc}
O &&&& U_s &\\
\hline
-U_s &&&& O &
\end{array}\right).
\]
Let $T_s'$ be the $(4s+1) \times (4s+1)$ matrix given by 
\[
T_s'\colonequals\left(\arraycolsep=1.4pt\def\arraystretch{1.5}
\begin{array}{cc|cc|ccccc}
V_s' && O_x && && U_{2s} &&\\
\hline
-(O_x)^T && 0 && && {}^{y^2}O &&\\
\hline
-U_{2s} && -(^{y^2}O)^T && && O &&
\end{array}\right).
\]
where $O_x$ is an appropriately sized column matrix with an $x$ at the bottom and zeroes elsewhere and ${}^{y^2}O$ is an appropriately sized row matrix with $y^2$ in the leftmost entry and zeroes elsewhere.
Let $I_s'$ be the ideal $((-1)^{i+1}\pf{i}{T_s'}))_{i=1,\ldots,4s+1}$, and let $J_s'$ be the ideal obtained by trimming the first $2s$ generators of $I_s'$.
\end{chunk}

The proofs of the next two results are similar to the proofs of \Cref{lem:OddEx} and \Cref{prop:OddEx} and are therefore omitted.

\begin{lemma}
The ideals $I_s'$ are Gorenstein ideals of grade 3.
\end{lemma}

\begin{proposition}\label{prop:EvenEx}
The ideal $J_s'$ has format $(1,6s-1,8s-1,2s+1)$ and is of class $\mathbf{G}(2s-1)$.
\end{proposition}

\section*{Acknowledgements}
We thank Lars Christensen for helpful conversations regarding the last section of this paper. We also thank the anonymous referee for their careful reading of the manuscript and
suggestions for improvement.

\appendix
\section{}
In this appendix we provide the full DG algebra structure on the resolution constructed in \Cref{thm:res}, complementing the products calculated in \Cref{thm:prod}.

Let $\alpha,\beta,l,s,\in\{1,2,3\}$ with $\alpha<\beta$, let $k$ be an integer between $1$ and $t$, and let $i,j$ be integers between $1$ and $m$. Before defining new constants, we recall the various constants that were defined on the main body of the paper. Let $T$ be a $m\times m$ skew-symmetric matrix with entries in the maximal ideal of $R$, then the constants $c_{i,j,l}$ are defined by the following equality 
\[
T_{j,i} = \sum_{l=1}^3 c_{i,j,l}z_l.
\]

The constants $c_{i,j,l}$ were used to define the following elements of $R$:

\begin{align*}
    d^k_{\alpha,\beta} &= \sum_{i=1}^m\sum_{r=1}^m \sigma_{i,k,r} c_{i,k,\beta} c_{r,k,\alpha} \pf{i,k,r}{T}, \\
    d^{k,i,j}_{\alpha,\beta} &= \sum_{r=1}^{m} \sigma_{i,j,r} \sum_{h=1}^m \sigma_{i,j,r,h,k} \pf{i,j,r,h,k}{T}c_{r,k,\beta}c_{h,k,\alpha},\\
    d_{\alpha,\beta}^{k,i,j,l} &=  z_ld_{\alpha,\beta}^{k,i,j},\quad\mathrm{for}\;k\neq i,\\
    d^{i,i,j,l}_{\alpha,\beta} &=  0\quad\mathrm{if}\;\{\alpha,\beta,l\}=\{1,2,3\},\\
    d^{i,i,j,\alpha}_{\alpha,\beta} &= \sum_{r=1}^m\sigma_{i,j,r}\pf{i,j,r}{T}c_{r,i,\beta},\\
    d^{i,i,j,\beta}_{\alpha,\beta} &= -\sum_{r=1}^m\sigma_{i,j,r}\pf{i,j,r}{T}c_{r,i,\alpha}.
\end{align*}

The new constants that we define in this appendix, which will be used in product $d$, are the following:

\begin{align*}
d_{\alpha,\beta}^{k,i,j,l,s}&\colonequals z_lz_sd_{\alpha,\beta}^{k,i,j},\quad\mathrm{for}\;k\neq i,j,\\
d_{\alpha,\beta}^{i,i,j,l,s}&\colonequals z_sd_{\alpha,\beta}^{i,i,j,l},\\
d_{\alpha,\beta}^{j,i,j,l,s}&\colonequals 0\quad\mathrm{if}\;\{\alpha,\beta,s\}=\{1,2,3\},\\
d_{\alpha,\beta}^{j,i,j,l,\alpha}&\colonequals z_l\sum_{r=1}^m\sigma_{i,j,r}\pf{i,j,r}{T}c_{r,j,\beta},\\
d_{\alpha,\beta}^{j,i,j,l,\beta}&\colonequals -z_l\sum_{r=1}^m\sigma_{i,j,r}\pf{i,j,r}{T}c_{r,j,\alpha}.
\end{align*}

We set $d_{\beta,\alpha}=-d_{\alpha,\beta}$ for any number of superscripts.

\begin{theorem}\label{thm:Appendix}
A DG algebra structure on the resolution constructed in \Cref{thm:res} is given by the following product rules

\noindent \underline{a. $F_1'\otimes F_1' \rightarrow F_2 \oplus (\oplus_{k=1}^t G_2^k)$}\\
\begin{equation*}
e_i\cdot e_j\colonequals e_i\cdot_Fe_j+\sum_{k=1}^t d^{k,i,j}_{1,2}v^k_{1,2}+d^{k,i,j}_{1,3}v^k_{1,3} + d^{k,i,j}_{2,3}v^k_{2,3},
\end{equation*}
for $t+1\leq i,j\leq m$.

\vspace{4mm}
\noindent\underline{b. $F_1'\otimes G_1^i \rightarrow F_2 \oplus (\oplus_{k=1}^t G_2^k)$}\\
\begin{equation*}
e_j\cdot u_l^i\colonequals z_le_i\cdot_Fe_j +\sum_{k=1}^t d^{k,i,j,l}_{1,2}v^k_{1,2}+d^{k,i,j,l}_{1,3}v^k_{1,3} + d^{k,i,j,l}_{2,3}v^k_{2,3},
\end{equation*}
for $1\leq i\leq t$ and $t+1\leq j\leq m$.

\vspace{4mm}
\noindent\underline{c. $G_1^i \otimes G_1^i \rightarrow F_2 \oplus (\oplus_{k=1}^t G_2^k)$}\\
\begin{equation*}
u_l^i\cdot u_s^i\colonequals -y_i v_{l,s}^i,
\end{equation*}
for $1\leq i\leq t$ and $l<s$.

\vspace{4mm}
\noindent\underline{d. $G_1^i \otimes G_1^j \rightarrow F_2 \oplus (\oplus_{k=1}^t G_2^k)$ with $i< j$}\\
\begin{equation*}
u_l^i \cdot u_s^j\colonequals z_l z_s e_i\cdot_Fe_j + \sum_{k=1}^t d^{k,i,j,l,s}_{1,2}v^k_{1,2}+d^{k,i,j,l,s}_{1,3}v^k_{1,3} + d^{k,i,j,l,s}_{2,3}v^k_{2,3},
\end{equation*}
for $ 1\leq i,j\leq t$.

\vspace{4mm}
\noindent \underline{e. $F_1' \otimes F_2 \rightarrow F_3 \oplus \left( \oplus_{k=1}^t G_3^k \right)$}\\
\begin{equation*}
e_i \cdot f_j\colonequals \begin{cases}
e_i\cdot_Ff_j, & t+1\leq j\leq m \\
- \sum_{r=1}^m c_{r,j,3} d^{j,i,r}_{1,2} w^j, & 1\leq j\leq t
\end{cases}
\end{equation*}
for $t+1\leq i \leq m$.

\vspace{4mm}
\noindent \underline{f. $F_1' \otimes G_2^i \rightarrow F_3 \oplus \left( \oplus_{k=1}^t G_3^k \right)$}\\
\begin{equation*}
e_j \cdot v_{\alpha,\beta}^i \colonequals (-1)^p\sum_{r=1}^m\sigma_{i,j,r}\pf{i,j,r}{T}c_{r,i,p}w^i,
\end{equation*}
with $p\in\{1,2,3\}\backslash\{\alpha,\beta\}$, $1\leq i \leq t$ and $t+1\leq j\leq m$.

\vspace{4mm}
\noindent \underline{g. $G_1^i \otimes G_2^i \rightarrow F_3 \oplus \left( \oplus_{k=1}^t G_3^k \right)$}\\
\begin{equation*}
u_l^i\cdot v_{\alpha,\beta}^i \colonequals -y_iu_l^i\cdot_{G^i}v^i_{\alpha,\beta},
\end{equation*}
for $1\leq i\leq t$.

\vspace{4mm}
\noindent\underline{h. $G_1^i \otimes G_2^j \rightarrow F_3 \oplus \left( \oplus_{k=1}^t G_3^k \right)$ with $i\neq j$}\\
\begin{equation*}
u_l^i \cdot v_{\alpha,\beta}^j\colonequals(-1)^{p+1}z_l\sum_{r=1}^m\sigma_{i,j,r}\pf{i,j,r}{T}c_{r,j,p}w^j
\end{equation*}
with $p\in\{1,2,3\}\backslash\{\alpha,\beta\}, 1\leq i,j\leq t$.

\vspace{4mm}
\noindent\underline{i. $G_1^i \otimes F_2 \rightarrow F_3 \oplus \left( \oplus_{k=1}^t G_3^k \right)$}\\
\begin{equation*}
u_l^i \cdot f_j\colonequals -z_le_i\cdot_Ff_j + (-1)^{l+1}\delta_{i,j}d^i_{\phi,\psi}w^i,
\end{equation*}
with $\{\phi,\psi,l\}=\{1,2,3\}$ and $\phi<\psi$, $1\leq i\leq t$ and $1\leq j\leq m$.
\end{theorem}

\begin{proof}
A proof of products $a$, $b$, and $e$ was given in \Cref{thm:prod}. Here we prove the remaining products.

\vspace{4mm}
\noindent\underline{$c$. $G_1^i \otimes G_1^i \rightarrow F_2 \oplus (\oplus_{k=1}^t G_2^k)$}\\
This product follows immediately from \cite[Theorem 3.2]{DGTrimmed}.

\vspace{4mm}
\noindent\underline{$d$. $G_1^i \otimes G_1^j \rightarrow F_2 \oplus (\oplus_{k=1}^t G_2^k)$ with $i\neq j$}\\
By \cite[Theorem 3.2]{DGTrimmed}, it suffices to show that
\begin{align}
\delta_2^i(d^{i,i,j,l,s}_{1,2}v^i_{1,2}+d^{i,i,j,l,s}_{1,3}v^i_{1,3} + d^{i,i,j,l,s}_{2,3}v^i_{2,3}) &=z_lz_sq_1^i(e_i\cdot_F e_j) + z_sy_ju_l^i,\label{eq:di}\\
\delta_2^j(d^{j,i,j,l,s}_{1,2}v^j_{1,2}+d^{j,i,j,l,s}_{1,3}v^j_{1,3} + d^{j,i,j,l,s}_{2,3}v^j_{2,3})&=z_lz_sq_1^j(e_i\cdot_F e_j) - z_ly_iu_s^j,\label{eq:dj}\\
\delta_2^k(d^{k,i,j,l,s}_{1,2}v^k_{1,2}+d^{k,i,j,l,s}_{1,3}v^k_{1,3} + d^{k,i,j,l,s}_{2,3}v^k_{2,3})&=z_lz_sq_1^k(e_i\cdot_Fe_j),\;\mathrm{for}\;k\neq i,j.\label{eq:dk}
\end{align}
We first prove \eqref{eq:di}. By the definitions of $d_{1,2}^{i,i,j,l,s}$, $d_{1,3}^{i,i,j,l,s}$ and $d_{2,3}^{i,i,j,l,s}$, the left hand side is equal to
\begin{align*}
&\delta_2^i(z_sd^{i,i,j,l}_{1,2}v^i_{1,2}+z_sd^{i,i,j,l}_{1,3}v^i_{1,3} + z_sd^{i,i,j,l}_{2,3}v^i_{2,3})\\
= &z_s\delta_2^i(d^{i,i,j,l}_{1,2}v^i_{1,2}+d^{i,i,j,l}_{1,3}v^i_{1,3} + d^{i,i,j,l}_{2,3}v^i_{2,3})\\
= &z_s(y_ju_l^i+z_lq_1^i(e_i\cdot_Fe_j))
\end{align*}
where the last equality was shown in the proof of \eqref{eq:bi} in product \textit{b}. Now we prove \eqref{eq:dj}. The left hand side is equal to 
\begin{align*}
(-d^{j,i,j,l,s}_{1,2}z_2 - d^{j,i,j,l,s}_{1,3}z_3)u^j_1 + (d^{j,i,j,l,s}_{1,2}z_1 -d^{j,i,j,l,s}_{2,3}z_3)u^j_2 + (d^{j,i,j,l,s}_{1,3}z_1+d^{j,i,j,l,s}_{2,3}z_2)u^j_3,
\end{align*}
while the right side is equal to 
\begin{align*}
    &z_l(z_sq_1^j(e_i\cdot_Fe_j) - y_iu_s^j)\\
    = &z_l\left(z_sq_1^j\left( \sum_{r=1}^{m} \sigma_{i,j,r} \pf{i,j,r}{T} f_r \right)- y_iu_s^j\right)\\
    = &z_l\left(z_s \sum_{r=1}^{m} \sigma_{i,j,r} \pf{i,j,r}{T} (c_{r,j,1}u_1^j + c_{r,j,2}u_2^j + c_{r,j,3}u_3^j)- y_iu_s^j\right).
\end{align*}
To prove that the coefficients of $u_1^j,u_2^j$ and $u_3^j$ match, we assume that $s=1$; the remaining cases are proved similarly. The coefficients of $u_2^j$ and $u_3^j$ match by the definition of the constants $d^{j,i,j,l,1}_{1,2}$, $d^{j,i,j,l,1}_{1,3}$ and $d^{j,i,j,l,1}_{2,3}$. Expanding $y_i$ using \eqref{eq:singlepf}, it follows that the coefficient of $u_1^j$ in $z_lz_sq_1^j(e_i\cdot_F e_j) - z_ly_iu_s^j$ is
\begin{samepage}
\begin{align*}
     &z_lz_1\sum_{r=1}^{m} \sigma_{i,j,r} \pf{i,j,r}{T} c_{r,j,1} - z_lz_1\sum_{r=1}^m \sigma_{i,j,r} \pf{i,j,r}{T} c_{r,j,1}\\
     - &z_lz_2\sum_{r=1}^m \sigma_{i,j,r} \pf{i,j,r}{T} c_{r,j,2} - z_lz_3\sum_{r=1}^m \sigma_{i,j,r} \pf{i,j,r}{T} c_{r,j,3} \\
     = - &z_lz_2\sum_{r=1}^m \sigma_{i,j,r} \pf{i,j,r}{T} c_{r,j,2} - z_lz_3\sum_{r=1}^m \sigma_{i,j,r} \pf{i,j,r}{T} c_{r,j,3}\\
     = -&d_{1,2}^{j,i,j,l,1}z_2 -d_{1,3}^{j,i,j,l,1}z_3.
\end{align*}
\end{samepage}
Finally we prove \eqref{eq:dk}. By the definitions of $d_{1,2}^{k,i,j,l,s}$, $d_{1,3}^{k,i,j,l,s}$ and $d_{2,3}^{k,i,j,l,s}$, the left hand side is equal to
\begin{align*}
&\delta_2^k(z_lz_sd^{k,i,j}_{1,2}v^k_{1,2}+z_lz_sd^{k,i,j}_{1,3}v^k_{1,3} + z_lz_sd^{k,i,j}_{2,3}v^k_{2,3})\\
= &z_lz_s\delta_2^k(d^{k,i,j}_{1,2}v^k_{1,2}+d^{k,i,j}_{1,3}v^k_{1,3} + d^{k,i,j}_{2,3}v^k_{2,3})\\
= &z_lz_sq_1^k(e_i\cdot_Fe_j)
\end{align*}
where last equality was shown in the proof of product \textit{a}.

\vspace{4mm}
\noindent \underline{$f$. $F_1' \otimes G_2^i \rightarrow F_3 \oplus \left( \oplus_{k=1}^t G_3^k \right)$}\\
Set 
\[
g_2^{\prime k} = \left(z_{\alpha} d^{k,i,j,\beta}_{1,2}-z_\beta d^{k,i,j,\alpha}_{1,2}\right) v^k_{1,2} + \left(z_{\alpha} d^{k,i,j,\beta}_{1,3}-z_\beta d^{k,i,j,\alpha}_{1,3}\right) v^k_{1,3} + \left(z_{\alpha}  d^{k,i,j,\beta}_{2,3}-z_\beta d^{k,i,j,\alpha}_{2,3}\right) v^k_{2,3}, \quad1\leq k\leq t.
\]
By \cite[Theorem 3.2]{DGTrimmed}, it suffices to show that
\begin{align}
\partial_1(e_j)v_{\alpha,\beta}^i - e_j\cdot \delta_2^i(v_{\alpha,\beta}^i) &= y_jv_{\alpha,\beta}^i - \sum_{k=1}^t g_2^{\prime k}\label{eq:f1}\\
g_2^{\prime k}&=0,\;\mathrm{if}\;k\neq i,\label{eq:f2}\\
\delta_3^i\left((-1)^p\sum_{r=1}^m\sigma_{i,j,r}\pf{i,j,r}{T}c_{r,i,p}w^i\right) &=  y_jv_{\alpha,\beta}^i - g_2^{\prime i}\label{eq:f3}.
\end{align}

We first prove \eqref{eq:f1}.
\begin{align*}
y_j v_{\alpha,\beta}^i &- e_j\cdot \left(z_{\alpha}u_{\beta}^i - z_{\beta}u_{\alpha}^i\right)\\
= y_j v_{\alpha,\beta}^i &- z_{\alpha}z_\beta e_i\cdot_Fe_j - z_{\alpha}\sum_{k=1}^t d^{k,i,j,\beta}_{1,2}v^k_{1,2}+d^{k,i,j,\beta}_{1,3}v^k_{1,3} + d^{k,i,j,\beta}_{2,3}v^k_{2,3}\\
&+ z_{\beta}z_\alpha e_i\cdot_Fe_j + z_\beta\sum_{k=1}^t d^{k,i,j,\alpha}_{1,2}v^k_{1,2}+d^{k,i,j,\alpha}_{1,3}v^k_{1,3} + d^{k,i,j,\alpha}_{2,3}v^k_{2,3}\\
= y_jv_{\alpha,\beta}^i &- z_{\alpha}\sum_{k=1}^t d^{k,i,j,\beta}_{1,2}v^k_{1,2}+d^{k,i,j,\beta}_{1,3}v^k_{1,3} + d^{k,i,j,\beta}_{2,3}v^k_{2,3} + z_\beta \sum_{k=1}^t d^{k,i,j,\alpha}_{1,2}v^k_{1,2}+d^{k,i,j,\alpha}_{1,3}v^k_{1,3} + d^{k,i,j,\alpha}_{2,3}v^k_{2,3}\\
= y_jv_{\alpha,\beta}^i &- \sum_{k=1}^t g_2^{\prime k},
\end{align*}
where the first equality follows from applying product $b$, the second equality follows from cancellation, and the third equality follows from definition of $g_2^{\prime k}$.

For \eqref{eq:f2} we notice that if $k\neq i$, then by the definition of $d_{-,-}^{k,i,j,-}$, we have that $g_2^{\prime k}$ is equal to
\begin{align*}
&z_{\alpha}z_{\beta}\left( d^{k,i,j}_{1,2}v^k_{1,2}+d^{k,i,j}_{1,3}v^k_{1,3} + d^{k,i,j}_{2,3}v^k_{2,3}\right) - z_\alpha z_\beta\left( d^{k,i,j}_{1,2}v^k_{1,2}+d^{k,i,j}_{1,3}v^k_{1,3} + d^{k,i,j}_{2,3}v^k_{2,3}\right) = 0.
\end{align*}

Now we prove \eqref{eq:f3}. We assume that $\alpha=1$ and $\beta=2$; the remaining two cases are proved similarly. The left side is equal to
\begin{align*}
&-\sum_{r=1}^m\sigma_{i,j,r}\pf{i,j,r}{T}c_{r,i,3}\delta_3^i(w^i)\\ 
= &-\sum_{r=1}^m\sigma_{i,j,r}\pf{i,j,r}{T}c_{r,i,3}(z_3 v_{1,2}^i - z_2 v_{1,3}^i + z_1 v_{2,3}^i)\\
= &-z_3\sum_{r=1}^m\sigma_{i,j,r}\pf{i,j,r}{T}c_{r,i,3}v_{1,2}^i +z_2\sum_{r=1}^m\sigma_{i,j,r}\pf{i,j,r}{T}c_{r,i,3}v_{1,3}^i -z_1\sum_{r=1}^m\sigma_{i,j,r}\pf{i,j,r}{T}c_{r,i,3}v_{2,3}^i,
\end{align*}
while the right hand side is equal to
\begin{align*}
&(y_j + z_{2}d^{i,i,j,1}_{1,2} - z_1 d^{i,i,j,2}_{1,2})v_{1,2}^i + (z_{2}d^{i,i,j,1}_{1,3} - z_1 d^{i,i,j,2}_{1,3})v_{1,3}^i + (z_{2}d^{i,i,j,1}_{2,3} - z_1 d^{i,i,j,2}_{2,3})v_{2,3}^i\\
= & (y_j + z_{2}d^{i,i,j,1}_{1,2} - z_1 d^{i,i,j,2}_{1,2})v_{1,2}^i + z_{2}d^{i,i,j,1}_{1,3}v_{1,3}^i - z_1 d^{i,i,j,2}_{2,3}v_{2,3}^i\\
= & -z_1\sum_{r=1}^m \sigma_{i,j,r} \pf{i,j,r}{T} c_{r,i,1}v_{1,2}^i -z_2\sum_{r=1}^m \sigma_{i,j,r} \pf{i,j,r}{T} c_{r,i,2}v_{1,2}^i -z_3\sum_{r=1}^m \sigma_{i,j,r} \pf{i,j,r}{T} c_{r,i,3}v_{1,2}^i\\
&+z_{2}\sum_{r=1}^m\sigma_{i,j,r}\pf{i,j,r}{T}c_{r,i,2}v_{1,2}^i + z_1 \sum_{r=1}^m\sigma_{i,j,r}\pf{i,j,r}{T}c_{r,i,1}v_{1,2}^i\\ 
&+ z_2\sum_{r=1}^m\sigma_{i,j,r}\pf{i,j,r}{T}c_{r,i,3}v_{1,3}^i -z_1\sum_{r=1}^m\sigma_{i,j,r}\pf{i,j,r}{T}c_{r,i,3}v_{2,3}^i\\
= & -z_3\sum_{r=1}^m \sigma_{i,j,r} \pf{i,j,r}{T} c_{r,i,3}v_{1,2}^i + z_2\sum_{r=1}^m\sigma_{i,j,r}\pf{i,j,r}{T}c_{r,i,3}v_{1,3}^i -z_1\sum_{r=1}^m\sigma_{i,j,r}\pf{i,j,r}{T}c_{r,i,3}v_{2,3}^i,
\end{align*}
where the first equality follows from the definition of $d_{1,3}^{i,i,j,2}$ and $d_{2,3}^{i,i,j,1}$, the second equality follows from \eqref{eq:singlepf} and from the definition of the constants $d_{-,-}^{i,i,j,-}$, and the third equality follows from cancellation.

\vspace{4mm}
\noindent \underline{$g$. $G_1^i \otimes G_2^i \rightarrow F_3 \oplus \left( \oplus_{k=1}^t G_3^k \right)$}\\
This product follows immediately from \cite[Theorem 3.2]{DGTrimmed}.

\vspace{4mm}
\noindent\underline{$h$. $G_1^i \otimes G_2^j \rightarrow F_3 \oplus \left( \oplus_{k=1}^t G_3^k \right)$ with $i\neq j$}\\
For $1\leq k\leq t$ set
\begin{align*}
g_2^{\prime k}&=\left(z_\alpha d_{1,2}^{k,i,j,l,\beta} - z_\beta d_{1,2}^{k,i,j,l,\alpha}\right) v_{1,2}^k + \left(z_\alpha d_{1,3}^{k,i,j,l,\beta} - z_\beta d_{1,3}^{k,i,j,l,\alpha}\right) v_{1,3}^k + \left(z_\alpha d_{2,3}^{k,i,j,l,\beta} - z_\beta d_{2,3}^{k,i,j,l,\alpha}\right) v_{2,3}^k.
\end{align*}
By \cite[Theorem 3.2]{DGTrimmed}, it suffices to show that
\begin{align}
\partial_1(u_l^i)v_{\alpha,\beta}^j - u_l^i\cdot \partial_2(v_{\alpha,\beta}^j) &= -y_iz_lv_{\alpha,\beta}^j - \sum_{k=1}^t g_{2}^{\prime k}\label{eq:h1}\\
g_2^{\prime k}, &=0 \;\mathrm{if}\;k\neq j,\label{eq:h2}\\
\delta_3^j\left((-1)^{p+1}z_l\sum_{r=1}^m\sigma_{i,j,r}\pf{i,j,r}{T}c_{r,i,p}w^j\right) &= g_2^{\prime j} + y_iz_lv_{\alpha,\beta}^j.\label{eq:h3}
\end{align}

We first prove \eqref{eq:h1}
\begin{align*}
&\partial_1(u_l^i)v_{\alpha,\beta}^j - u_l^i\cdot \partial_2(v_{\alpha,\beta}^j)\\
= &-y_iz_l v_{\alpha,\beta}^j - z_{\alpha} u_l^i \cdot u_{\beta}^j + z_{\beta}u_l^i \cdot u_{\alpha}^j\\
= &-y_iz_l v_{\alpha,\beta}^j - z_{\alpha}z_lz_\beta e_i\cdot_Fe_j- z_\alpha\sum_{k=1}^t d_{1,2}^{k,i,j,l,\beta}v_{1,2} + d_{1,3}^{k,i,j,l,\beta}v_{1,3} + d_{2,3}^{k,i,j,l,\beta}v_{2,3}\\
&+ z_{\beta}z_lz_\alpha e_i\cdot_Fe_j + z_\beta\sum_{k=1}^t d_{1,2}^{k,i,j,l,\alpha}v_{1,2} + d_{1,3}^{k,i,j,l,\alpha}v_{1,3} + d_{2,3}^{k,i,j,l,\alpha}v_{2,3}\\
= &-y_iz_l v_{\alpha,\beta}^j - z_\alpha \sum_{k=1}^t d_{1,2}^{k,i,j,l,\beta}v_{1,2} + d_{1,3}^{k,i,j,l,\beta}v_{1,3} + d_{2,3}^{k,i,j,l,\beta}v_{2,3}\\
&+ z_\beta\sum_{k=1}^t d_{1,2}^{k,i,j,l,\alpha}v_{1,2} + d_{1,3}^{k,i,j,l,\alpha}v_{1,3} + d_{2,3}^{k,i,j,l,\alpha}v_{2,3}\\
= &-y_iz_lv_{\alpha,\beta}^j - \sum_{k=1}^t g_{2}^{\prime k},
\end{align*}
where the first equality follows from distribution, the second equality follows from applying product $d$, the third equality follows from cancellation, and the fourth equality follows from the definition of $g_2^{\prime k}$.
For \eqref{eq:h2} we notice that if $k\neq j$, then $g_2^{\prime k}$ is equal to
\begin{align*}
\left(z_\alpha z_\beta d_{1,2}^{k,i,j,l} - z_\alpha z_\beta d_{1,2}^{k,i,j,l}\right) v_{1,2}^k + \left(z_\alpha z_\beta d_{1,3}^{k,i,j,l} - z_\alpha z_\beta d_{1,3}^{k,i,j,l}\right) v_{1,3}^k + \left(z_\alpha z_\beta d_{2,3}^{k,i,j,l} - z_\alpha z_\beta d_{2,3}^{k,i,j,l}\right) v_{2,3}^k = 0.
\end{align*}

The proof of \eqref{eq:h3} is similar to the proof of \eqref{eq:f3} and is therefore omitted.

\vspace{4mm}
\noindent\underline{$i$. $G_1^i \otimes F_2 \rightarrow F_3 \oplus \left( \oplus_{k=1}^t G_3^k \right)$}\\
Set 
\begin{align*}
g_2^k &= \sum_{r=t+1}^m T_{r,j} \left(d^{k,i,r,l}_{1,2}v^k_{1,2}+d^{k,i,r,l}_{1,3}v^k_{1,3} + d^{k,i,r,l}_{2,3}v^k_{2,3}\right), \quad1\leq k\leq t,\\
g_2'^{k,h} &= \sum_{s=1}^3 c_{j,k,s} \left(d^{h,i,k,l,s}_{1,2}v^h_{1,2}+d^{h,i,k,l,s}_{1,3}v^h_{1,3} + d^{h,i,k,l,s}_{2,3}v^h_{2,3}\right), \quad1\leq k,h\leq t.
\end{align*}
By \cite[Theorem 3.2]{DGTrimmed}, it suffices to show that
\begin{align}
&\partial_1(u_l^i)f_j - u_l^i\cdot \partial_2(f_j) = -z_l D_3(e_i \cdot_F f_j) - y_iu_l^i\cdot_{G^i} q_1^i(f_j) + g_2^i + \sum_{\underset{k\neq i}{k=1}}^t \left(g_2^k + \sum_{h=1}^t g_2'^{k,h} \right).\label{eq:i1}\\
&g_2^k + \sum_{\underset{h\neq i}{h=1}}^t g_2'^{h,k} + z_lq_2^k(e_i\cdot_Ff_j) = 0.\label{eq:i2}\\
&g_2^i + \sum_{\underset{h\neq i}{h=1}}^t g_2'^{h,i} + z_lq_2^i(e_i\cdot_Ff_j) -y_i u_l^i\cdot_{G^i} q_1^i(f_j) = \delta_3^i \left((-1)^{l+1}\delta_{i,j}d^i_{\phi,\psi}w^i\right).\label{eq:i3}
\end{align}
We first prove \eqref{eq:i1}.
\begin{align*}
&\partial_1(u_l^i)f_j - u_l^i\cdot \partial_2(f_j) \\
= &-y_iz_lf_j - u_l^i\cdot\left(D_2'(f_j)-\sum_{k=1}^t q_1^k(f_j)\right)\\
= &-y_iz_lf_j - u_l^i\cdot \sum_{r=t+1}^m T_{r,j}e_r + u_l^i \cdot \sum_{k=1}^t \sum_{s=1}^3 c_{j,k,s} u_s^k\\
= &-y_iz_lf_j + \sum_{r=t+1}^m T_{r,j}z_le_i\cdot_Fe_r + \sum_{r=t+1}^m T_{r,j}\sum_{k=1}^t \left(d^{k,i,r,l}_{1,2}v^k_{1,2}+d^{k,i,r,l}_{1,3}v^k_{1,3} + d^{k,i,r,l}_{2,3}v^k_{2,3}\right)\\ 
&+ \sum_{\underset{k\neq i}{k=1}}^t \sum_{s=1}^3 c_{j,k,s} u_l^i\cdot u_s^k + \sum_{s=1}^3 c_{j,i,s} u_l^i\cdot u_s^i \\
= &-y_iz_lf_j + \sum_{r=t+1}^m T_{r,j}z_le_i\cdot_Fe_r + \sum_{r=t+1}^m T_{r,j}\sum_{k=1}^t \left(d^{k,i,r,l}_{1,2}v^k_{1,2}+d^{k,i,r,l}_{1,3}v^k_{1,3} + d^{k,i,r,l}_{2,3}v^k_{2,3}\right)\\
&+ \sum_{\underset{k\neq i}{k=1}}^t \sum_{s=1}^3 c_{j,k,s} z_lz_se_i\cdot_Fe_k + \sum_{\underset{k\neq i}{k=1}}^t \sum_{s=1}^3 c_{j,k,s} \sum_{h=1}^t \left(d^{h,i,k,l,s}_{1,2}v^h_{1,2}+d^{h,i,k,l,s}_{1,3}v^h_{1,3} + d^{h,i,k,l,s}_{2,3}v^h_{2,3}\right) - \sum_{s=1}^3 c_{j,i,s} y_iv_{l,s}^i\\
=& -y_iz_lf_j + \sum_{r=t+1}^m T_{r,j}z_le_i\cdot_Fe_r+ \sum_{r=t+1}^m T_{r,j}\sum_{k=1}^t \left(d^{k,i,r,l}_{1,2}v^k_{1,2}+d^{k,i,r,l}_{1,3}v^k_{1,3} + d^{k,i,r,l}_{2,3}v^k_{2,3}\right) + \sum_{\underset{r\neq i}{r=1}}^t T_{r,j} z_le_i\cdot_Fe_r\\
&+ \sum_{\underset{k\neq i}{k=1}}^t \sum_{s=1}^3 c_{j,k,s} \sum_{h=1}^t \left(d^{h,i,k,l,s}_{1,2}v^h_{1,2}+d^{h,i,k,l,s}_{1,3}v^h_{1,3} + d^{h,i,k,l,s}_{2,3}v^h_{2,3}\right) - \sum_{s=1}^3 c_{j,i,s} y_iv_{l,s}^i\\
= &z_l\left(\sum_{r=1}^m T_{r,j}e_i\cdot_Fe_r - y_if_j\right) + \sum_{r=t+1}^m T_{r,j}\sum_{k=1}^t \left(d^{k,i,r,l}_{1,2}v^k_{1,2}+d^{k,i,r,l}_{1,3}v^k_{1,3} + d^{k,i,r,l}_{2,3}v^k_{2,3}\right)\\
&+ \sum_{\underset{k\neq i}{k=1}}^t \sum_{s=1}^3 c_{j,k,s} \sum_{h=1}^t d^{h,i,k,l,s}_{1,2}v^h_{1,2}+d^{h,i,k,l,s}_{1,3}v^h_{1,3} + d^{h,i,k,l,s}_{2,3}v^h_{2,3} - \sum_{s=1}^3 c_{j,i,s} y_iv_{l,s}^i\\
=& -z_l D_3(e_i \cdot_F f_j) + \sum_{k=1}^t g_2^k +\sum_{\underset{k\neq i}{k=1}}^t\sum_{h=1}^t g_2'^{k,h} - y_iu_l^i\cdot_{G^i} q_1^i(f_j)\\
=& -z_l D_3(e_i \cdot_F f_j) - y_iu_l^i\cdot_{G^i} q_1^i(f_j) + g_2^i + \sum_{\underset{k\neq i}{k=1}}^t \left(g_2^k + \sum_{h=1}^t g_2'^{k,h} \right)
\end{align*}
where the first and second equalities follow from applying the differential maps, the third equality follows from applying product $b$ and separating the case $k=i$, the fourth equality follows from applying products $c$ and $d$, the fifth equality follows from reindexing and \eqref{eq:Tji}, the sixth equality follows from factoring out a $z_l$, the seventh equality follows from the definitions of $D_3$, $g_2^k$, $g_2^{\prime k,h}$, and the product in the Koszul complex, and the eighth equality is obvious.
Now we show \eqref{eq:i2}. Notice that if we fix $\alpha<\beta$, then the coefficient of $v^k_{\alpha,\beta}$ on the left hand side of \eqref{eq:i2}
is equal to
\begin{align*}
&\sum_{r=t+1}^m T_{r,j} d^{k,i,r,l}_{\alpha,\beta} + \sum_{\underset{h\neq i}{h=1}}^t \sum_{s=1}^3 c_{j,h,s} d^{k,i,h,l,s}_{\alpha,\beta} + \delta_{i,j}z_ld^k_{\alpha,\beta}\\
= &z_l\sum_{r=t+1}^m T_{r,j} d^{k,i,r}_{\alpha,\beta} + \sum_{\underset{h\neq i,k}{h=1}}^t \sum_{s=1}^3 c_{j,h,s} d^{k,i,h,l,s}_{\alpha,\beta} + \sum_{s=1}^3 c_{j,k,s} d^{k,i,k,l,s}_{\alpha,\beta} + \delta_{i,j}z_ld^k_{\alpha,\beta}\\
= &z_l\sum_{r=t+1}^m T_{r,j} d^{k,i,r}_{\alpha,\beta} + z_l \sum_{\underset{h\neq i,k}{h=1}}^t \sum_{s=1}^3 c_{j,h,s}z_s d^{k,i,h}_{\alpha,\beta} + c_{j,k,\alpha} d^{k,i,k,l,\alpha}_{\alpha,\beta} + c_{j,k,\beta} d^{k,i,k,l,\beta}_{\alpha,\beta} + \delta_{i,j}z_ld^k_{\alpha,\beta}\\
= &z_l\sum_{r=t+1}^m T_{r,j} d^{k,i,r}_{\alpha,\beta} + z_l \sum_{\underset{h\neq i,k}{h=1}}^t T_{h,j} d^{k,i,h}_{\alpha,\beta} + z_l\sum_{r=1}^m \sigma_{i,k,r} \pf{i,k,r}{T} c_{j,k,\alpha}c_{r,k,\beta}\\
&- z_l\sum_{r=1}^m \sigma_{i,k,r} \pf{i,k,r}{T} c_{r,k,\alpha}c_{j,k,\beta} + \delta_{i,j}z_l d^k_{\alpha,\beta}\\
= &z_l\sum_{r=1}^m T_{r,j} d^{k,i,r}_{\alpha,\beta} + z_l\sum_{r=1}^m \sigma_{i,k,r} \pf{i,k,r}{T} (c_{j,k,\alpha}c_{r,k,\beta} - c_{r,k,\alpha}c_{j,k,\beta}) + \delta_{i,j}z_l d^k_{\alpha,\beta}\\
= &z_l \left(\sum_{r=1}^m T_{r,j} d^{k,i,r}_{\alpha,\beta} + \sum_{r=1}^m \sigma_{i,k,r} \pf{i,k,r}{T} (c_{j,k,\alpha}c_{r,k,\beta} - c_{r,k,\alpha}c_{j,k,\beta}) + \delta_{i,j}d^k_{\alpha,\beta}\right)\\
= & z_l\cdot0\\
=& 0
\end{align*}
where the first equality follows from the definition of $d_{\alpha,\beta}^{k,i,r,l}$ and separating the case $h=k$, the second equality follows from definition of $d_{\alpha,\beta}^{k,i,h,l,s}$ and expanding the third term, the third equality follows from \eqref{eq:Tji} and the definition of $d_{\alpha,\beta}^{k,i,k,l,-}$, the fourth equality follows from combining summations, the fifth equality follows from factoring out $z_l$, and the sixth equality follows from the computations in the proof of \eqref{eq:e2}.
Finally we show \eqref{eq:i3}. We fix $\alpha<\beta$ and we calculate the coefficient of $v_{\alpha,\beta}^i$ on both sides of \eqref{eq:i3}. There are three cases to consider. First, if $\{l,\alpha,\beta\}=\{1,2,3\}$, then the coefficient of $v^i_{\alpha,\beta}$ on the left hand side of \eqref{eq:i3} is 
\begin{align*}
\sum_{r=t+1}^m T_{r,j}d^{i,i,r,l}_{\alpha,\beta} + \sum_{\underset{h\neq i}{h=1}}^t \sum_{s=1}^3 c_{j,h,s} z_s d_{\alpha,\beta}^{i,i,h,l} + \delta_{i,j} z_l d_{\alpha,\beta}^i,
\end{align*}
as $u_l^i\cdot_{G^i}u_l^i=0$, $u_l^i\cdot_{G^i}u_{\alpha}^i=v_{l,\alpha}^i$ and $u_l^i\cdot_{G^i}u_{\beta}^i=v_{l,\beta}^i$. Notice that since $d^{i,i,r,l}_{\alpha,\beta} = 0$ and $d^{i,i,h,l}_{\alpha,\beta} = 0$, the above display is equal to $\delta_{i,j}z_ld_{\alpha,\beta}^i$, which coincides with the coefficient of $v_{\alpha,\beta}^i$ on the right hand side of \eqref{eq:i3}. Now if $l=\alpha$ and $\{\alpha,\beta,\epsilon\}=\{1,2,3\}$, then the coefficient of $v^i_{\alpha,\beta}$ in the left hand side of \eqref{eq:i3} is
\begin{align*}
&\sum_{r=t+1}^m T_{r,j}d^{i,i,r,l}_{\alpha,\beta} + \sum_{h=1}^t \sum_{s=1}^3 c_{j,h,s}d^{i,i,h,l,s}_{\alpha,\beta} + \delta_{i,j} z_l d^i_{\alpha,\beta} - y_i c_{j,i,\beta}\\
= &\sum_{r=t+1}^m T_{r,j}d^{i,i,r,l}_{\alpha,\beta} + \sum_{h=1}^t \sum_{s=1}^3 z_sc_{j,h,s}d^{i,i,h,l}_{\alpha,\beta} + \delta_{i,j} z_l d^i_{\alpha,\beta} - y_i c_{j,i,\beta}\\
= &\sum_{r=t+1}^m T_{r,j}d^{i,i,r,l}_{\alpha,\beta} + \sum_{h=1}^t T_{h,j} d^{i,i,h,l}_{\alpha,\beta} + \delta_{i,j} z_l d^i_{\alpha,\beta} - y_i c_{j,i,\beta}\\
= &\sum_{r=t+1}^m T_{r,j}d^{i,i,r,l}_{\alpha,\beta} + \sum_{r=1}^t T_{r,j} d^{i,i,r,l}_{\alpha,\beta} + \delta_{i,j} z_l d^i_{\alpha,\beta} - y_i c_{j,i,\beta}\\
= &\sum_{r=1}^m T_{r,j}d^{i,i,r,l}_{\alpha,\beta} + \delta_{i,j} z_l d^i_{\alpha,\beta} - y_i c_{j,i,\beta}\\
= & \sum_{r=1}^m T_{r,j} \sum_{h=1}^m \sigma_{i,r,h} \pf{i,r,h}{T} c_{h,i,\beta} + \delta_{i,j} z_l d^i_{\alpha,\beta} - y_i c_{j,i,\beta}\\
= & \sum_{r=1}^m T_{r,j} \sum_{\underset{h\neq j}{h=1}}^m \sigma_{i,r,h} \pf{i,r,h}{T} c_{h,i,\beta} + \sum_{r=1}^m T_{r,j} \sigma_{i,r,j} \pf{i,r,j}{T} c_{j,i,\beta} + \delta_{i,j} z_l d^i_{\alpha,\beta} - y_i c_{j,i,\beta}\\
= &\sum_{r=1}^m T_{r,j} \sum_{\underset{h\neq j}{h=1}}^m \sigma_{i,r,h} \pf{i,r,h}{T} c_{h,i,\beta} + \sum_{r=1}^m T_{j,r} \sigma_{i,j,r} \pf{i,j,r}{T} c_{j,i,\beta} + \delta_{i,j} z_l d^i_{\alpha,\beta} - y_i c_{j,i,\beta}\\
= & \sum_{r=1}^m T_{r,j} \sum_{\underset{h\neq j}{h=1}}^m \sigma_{i,r,h} \pf{i,r,h}{T} c_{h,i,\beta} + y_i c_{j,i,\beta} + \delta_{i,j} z_l d^i_{\alpha,\beta} - y_i c_{j,i,\beta}\\
= & \sum_{r=1}^m T_{r,j} \sum_{\underset{h\neq j}{h=1}}^m \sigma_{i,r,h} \pf{i,r,h}{T} c_{h,i,\beta} + \delta_{i,j} z_l d^i_{\alpha,\beta}
\end{align*}
where the first equality follows from definition of $d_{\alpha,\beta}^{i,i,h,l,s}$, the second equality follows from \eqref{eq:Tji}, the third equality follows from reindexing, the fourth equality follows from combining summations, the fifth equality follows from definition of $d_{\alpha,\beta}^{i,i,r,l}$, the sixth equality follows from separating the case $h=j$, the seventh equality follows since $T$ is skew-symmetric and from \eqref{eq:Sig3Rel}, the eighth equality follows from \eqref{eq:y_i}, and the ninth equality follows from cancellation.
If $i\neq j$, then by \eqref{eq:3pf0},
\[
\sum_{r=1}^m T_{r,j} \sum_{\underset{h\neq j}{h=1}}^m \sigma_{i,r,h} \pf{i,r,h}{T} c_{h,i,\beta}=0,
\]
showing \eqref{eq:i3} when $l=\alpha$ and $i\neq j$.
If $i=j$, then we have 
\begin{align*}
&\sum_{r=1}^m T_{r,i} \sum_{\underset{h\neq i}{h=1}}^m \sigma_{i,r,h} \pf{i,r,h}{T} c_{h,i,\beta} + z_{\alpha} d^i_{\alpha,\beta}\\
= &\sum_{r=1}^m T_{r,i} \sum_{\underset{h\neq i}{h=1}}^m \sigma_{i,r,h} \pf{i,r,h}{T} c_{h,i,\beta} + z_{\alpha} \sum_{h=1}^m \sum_{r=1}^m \sigma_{h,i,r} \pf{h,i,r}{T} c_{r,i,\alpha} c_{h,i,\beta}\\
= &-\sum_{r=1}^m \sum_{s=1}^3 c_{r,i,s}z_s \sum_{\underset{h\neq i}{h=1}}^m \sigma_{i,r,h} \pf{i,r,h}{T} c_{h,i,\beta} + z_{\alpha} \sum_{h=1}^m \sum_{r=1}^m \sigma_{h,i,r} \pf{h,i,r}{T} c_{r,i,\alpha} c_{h,i,\beta}\\
= &-z_{\alpha}\sum_{r=1}^m \sum_{\underset{h\neq i}{h=1}}^m \sigma_{i,r,h} \pf{i,r,h}{T} c_{r,i,\alpha}c_{h,i,\beta} - z_{\beta}\sum_{r=1}^m \sum_{\underset{h\neq i}{h=1}}^m \sigma_{i,r,h} \pf{i,r,h}{T} c_{r,i,\beta}c_{h,i,\beta}\\ 
&- z_\epsilon\sum_{r=1}^m \sum_{\underset{h\neq i}{h=1}}^m \sigma_{i,r,h} \pf{i,r,h}{T} c_{r,i,\epsilon}c_{h,i,\beta} + z_{\alpha} \sum_{h=1}^m \sum_{r=1}^m \sigma_{h,i,r} \pf{h,i,r}{T} c_{r,i,\alpha} c_{h,i,\beta}\\
= &-z_{\beta}\sum_{\underset{r\neq i}{r=1}}^m \sum_{\underset{h\neq i}{h=1}}^m \sigma_{i,r,h} \pf{i,r,h}{T} c_{r,i,\beta}c_{h,i,\beta} - z_\epsilon\sum_{r=1}^m \sum_{\underset{h\neq i}{h=1}}^m \sigma_{i,r,h} \pf{i,r,h}{T} c_{r,i,\epsilon}c_{h,i,\beta}\\
= &- z_\epsilon\sum_{\underset{r\neq i}{r=1}}^m \sum_{\underset{h\neq i}{h=1}}^m \sigma_{i,r,h} \pf{i,r,h}{T} c_{r,i,\epsilon}c_{h,i,\beta}\\
= &z_\epsilon\sum_{\underset{r\neq i}{r=1}}^m \sum_{\underset{h\neq i}{h=1}}^m \sigma_{r,i,h} \pf{r,i,h}{T} c_{r,i,\epsilon}c_{h,i,\beta}\\
= &z_\epsilon d^i_{\beta,\epsilon}
\end{align*}
where the first equality follows from definition of $d_{\alpha,\beta}^i$, the second equality follows from \eqref{eq:Tji}, the third equality follows from separating the cases $s=\alpha$, $s=\beta$ and $s=\epsilon$, the fourth equality follows from cancellation, the fifth equality follows from \eqref{eq:Sig3Rel} and cancellation, the sixth equality follows from \eqref{eq:Sig3Rel}, and the seventh equality follows from definition of $d^i_{\beta,\epsilon}$.
The case $l=\beta$ follows similarly and is therefore omitted.
Hence the coefficient of $v_{\alpha,\beta}^i$ is $(-1)^{l+1}\delta_{i,j}d^i_{\phi,\psi}$ when $\{\phi,\psi,l\}=\{1,2,3\}$ and $\phi<\psi$.
\end{proof}

\bibliographystyle{amsplain}
\bibliography{biblio}

\providecommand{\bysame}{\leavevmode\hbox to3em{\hrulefill}\thinspace}
\providecommand{\MR}{\relax\ifhmode\unskip\space\fi MR }
\providecommand{\MRhref}[2]{%
  \href{http://www.ams.org/mathscinet-getitem?mr=#1}{#2}
}
\providecommand{\href}[2]{#2}
\begin{thebibliography}{10}

\bibitem{Enum}
Martin Aigner, \emph{A course in enumeration}, Graduate Texts in Mathematics,
  vol. 238, Springer, Berlin, 2007.

\bibitem{small}
Luchezar~L. Avramov, \emph{Small homomorphisms of local rings}, J. Algebra
  \textbf{50} (1978), no.~2, 400--453.

\bibitem{Avramov2012}
\bysame, \emph{A cohomological study of local rings of embedding codepth 3}, J.
  Pure Appl. Algebra \textbf{216} (2012), no.~11, 2489--2506.

\bibitem{AKM}
Luchezar~L. Avramov, Andrew~R. Kustin, and Matthew Miller, \emph{Poincar\'{e}
  series of modules over local rings of small embedding codepth or small
  linking number}, J. Algebra \textbf{118} (1988), no.~1, 162--204.

\bibitem{BuchEisen}
David~A. Buchsbaum and David Eisenbud, \emph{Algebra structures for finite free
  resolutions, and some structure theorems for ideals of codimension {$3$}},
  Amer. J. Math. \textbf{99} (1977), no.~3, 447--485.

\bibitem{CVExamples}
Lars~Winther Christensen and Oana Veliche, \emph{Local rings of embedding
  codepth 3. examples}, Algebras and Representation Theory \textbf{17} (2014),
  no.~1, 121--135.

\bibitem{trimming}
Lars~Winther Christensen, Oana Veliche, and Jerzy Weyman, \emph{Trimming a
  {G}orenstein ideal}, Journal of Commutative Algebra \textbf{11} (2019),
  no.~3, 325--339.

\bibitem{Linkage}
\bysame, \emph{Linkage classes of grade 3 perfect ideals}, J. Pure Appl.
  Algebra \textbf{224} (2020), no.~6, 106185, 29.

\bibitem{aci}
\bysame, \emph{Three takes on almost complete intersection ideals of grade 3},
  Commutative Algebra: Expository Papers Dedicated to David Eisenbud on the
  Occasion of his 75th Birthday, Springer, 2021, pp.~219--281.

\bibitem{Knuth}
Donald~E. Knuth, \emph{Overlapping {P}faffians}, Electron. J. Combin.
  \textbf{3} (1996), no.~2, Research Paper 5, approx. 13, The Foata
  Festschrift.

\bibitem{BlockDet}
John~R. Silvester, \emph{Determinants of block matrices}, The Mathematical
  Gazette \textbf{84} (2000), no.~501, 460–467.

\bibitem{DGTrimmed}
Keller VandeBogert, \emph{D{G} structure on length 3 trimming complexes and
  applications to tor algebras}, Journal of Pure and Applied Algebra
  \textbf{226} (2022), no.~9, 107053.

\bibitem{Weyman}
Jerzy Weyman, \emph{On the structure of free resolutions of length {$3$}}, J.
  Algebra \textbf{126} (1989), no.~1, 1--33.

\end{thebibliography}
\end{document}